\documentclass[onefignum,onetabnum]{siamart250211}

\usepackage{amsfonts}
\usepackage{amssymb}
\usepackage{graphicx}
\usepackage{xspace}
\usepackage{xfrac}
\usepackage{relsize}
\usepackage{mathtools}
\usepackage{bm}
\usepackage{chemformula}
\usepackage{booktabs}
\usepackage[noend]{algpseudocode}
\usepackage[shortlabels]{enumitem}
\usepackage{xcolor}
\usepackage{multirow}
\usepackage{svg}
\usepackage{subcaption}
\usepackage{rotating}
\usepackage{overpic}
\usepackage{mathrsfs}
\usepackage{tikz}
\usepackage{xr-hyper}
\usetikzlibrary{arrows.meta}

\usepackage{subcaption}
\usepackage{todonotes}
\usepackage{diagbox}

\ifpdf
  \DeclareGraphicsExtensions{.eps,.pdf,.png,.jpg}
\else
  \DeclareGraphicsExtensions{.eps}
\fi

\newsiamremark{remark}{Remark}
\newsiamremark{hypothesis}{Hypothesis}
\crefname{hypothesis}{Hypothesis}{Hypotheses}
\newsiamthm{claim}{Claim}
\newsiamremark{fact}{Fact}
\crefname{fact}{Fact}{Facts}

\headers{Chebyshev-Tucker tensor format}{P. Benner, B. N. Khoromskij, V. Khoromskaia, B. Sun}

\title{A hybrid Chebyshev-Tucker tensor format for approximation of multivariate functions\thanks{
  Submitted to the editors DATE.
  }}

\author{Peter Benner \thanks{
  Max Planck Institute for Dynamics of Complex Systems, Sandtorstr. 1, D-39106 Magdeburg, Germany (\email{benner@mpi-magdeburg.mpg.de},\email{bsun@mpi-magdeburg.mpg.de}).
  }
 \and Boris N. Khoromskij \footnotemark[2] \thanks{Max Planck Institute for
        Mathematics in the Sciences, Inselstr.~22-26, D-04103 Leipzig, Germany (\email{bokh@mis.mpg.de}, \email{vekh@mis.mpg.de})
        }
        \and Venera Khoromskaia \footnotemark[3]
        \and Bonan Sun \footnotemark[2]
        }

\newcommand{\R}{\mathbb{R}}

\newcommand{\Z}{\mathbb{Z}}
\newcommand{\e}{\mathrm{e}}
\newcommand{\dd}{\operatorname{d}}

\DeclareMathOperator{\err}{err}

\colorlet{dkgreen}{green!60!black}

\ifpdf
\hypersetup{
  pdftitle={ChebTuck format},
  pdfauthor={P. Benner, B. N. Khoromskij, V. Khoromskaia, B. Sun}
}
\fi

\begin{document}

\maketitle

\begin{abstract}
  We introduce and analyze a mesh-free two-level hybrid Chebyshev-Tucker tensor representation for approximating multivariate functions, which combines tensor-product Chebyshev interpolation with the low-rank Tucker decomposition of the tensor of Chebyshev coefficients.
  This construction allows to avoid the expensive rank-structured grid-based approximation of function-related tensors on large spatial grids, while benefiting from the Tucker decomposition of the moderate-sized core tensor of Chebyshev coefficients.
  Thus, we can compute the nearly optimal Tucker decomposition of the 3D function with controllable accuracy $\varepsilon >0$ without discretizing the function on a full fine grid in the domain, but only using its values at a small set of Chebyshev nodes computed either from the explicit analytic expression of the target function or from its data-sparse representation in a rank-structured tensor format
  with moderate rank parameter.
  Finally, we can represent the function in the algebraic Tucker format with optimal $\varepsilon$-rank on an arbitrarily large 3D tensor grid in the computational domain by discretizing the Chebyshev polynomials on that grid.
  The rank parameters of the nonlinear Tucker-ALS decomposition of the coefficient tensor can be much smaller than the polynomial degrees of the initial Chebyshev linear interpolation in the function independent polynomial basis set.
  It is shown that our techniques can be gainfully applied to the long-range part of the singular electrostatic potential of multi-particle systems represented on a fine grid in the range-separated (RS) tensor format. The resulting low-rank tensor can be used in applications to numerical PDE, in computational quantum chemistry, as well for calculation of various physical quantities of  multi-particle systems.
  We provide error and complexity estimates and demonstrate the computational efficiency of the proposed techniques on challenging examples, including the collective electrostatic potential for large bio-molecular systems and lattice-type compounds.

\end{abstract}



\section{Introduction}
Tensor numerical methods (TNM) intrinsically simplify the numerical solution of multi-dimensional problems to one-dimensional calculations.
In recent years, TNM  are proven to be a powerful tool for solving numerically intensive multi-dimensional problems in scientific computing arising in computational quantum chemistry \cite{Khor_bookQC_2018}, multi-dimensional steady state and dynamical PDE-driven problems, data science, stochastic simulations, optimal control problems, etc., see, e.g., \cite{Bach:23,Khor-book_2018} for a detailed discussion and related references.
The benefits of TNM for $d$-dimensional problems are due to the reduction of both computational and storage complexity to merely linear scaling in $d$
based on the representation of functions and operators in low-parametric rank structured tensor formats.

TNM originated as bridging of approximation theory for multi-dimensional functions and operators from one side and the modern multilinear algebra techniques from the other side.
The construction of efficient TNM is guided by the trade-off between approximation accuracy and computational complexity of the chosen rank-structured approximation algorithms, while taking into account the compatibility of the resultant tensor representations with available multilinear algebra.
The total numerical complexity of tensor approximations is a sum of the high computational (and storage) cost for (a) the algebraic operations needed for the low-rank approximation of multi-dimensional functions and operators in the conventional canonical (CP) \cite{Hitch:27}, Tucker \cite{Tuck:1966}, and rather recent tensor train (TT) \cite{Osel_TT:11}, quantized TT (QTT) \cite{KhQuant:09} or range-separated (RS) \cite{benner2018range} tensor formats, and (b) the cost of function evaluations required for implementation of the tensor algorithms.
Thereafter, the overall approximation accuracy can be controlled by rank parameters as well as the precision of the chosen functional approximations (given the storage budget, accuracy is limited by the data regularity). A number of recent developments confirm efficiency of TNM in different research fields~\cite{Usch:16,Gavini:21,Gav:24,Schwab:22,Solo:21,Kilmer:22,MSchneid:22,Yser:24}.

TNM are usually applied to grid-based rank-structured tensor approximation of target functions and operators discretized on large $n^{\otimes d}$ spatial grids in $\mathbb{R}^d$, $d \geq 2$, for the numerical solution of large scale scientific computing problems~\cite{Khor-book_2018,Khor_bookQC_2018,RSDock:23}.
In this direction the sparse grid methods, stochastic homogenization techniques and UQ approaches also provide powerful computational tools for special classes
of problems.
Alternatively, classical mesh-free approximation methods for functions are based on their representation in problem-independent basis sets, such as trigonometric functions (Fourier), (piecewise) polynomials, sinc functions or Gaussians~\cite{Stenger:93,trefethen2019,Braess:86,mason2002chebyshev,Schwab:22,Gao:22}.
The fast multipole method provides another example of a mesh-free approach for certain kernel-based problems~\cite{Green:18}. Furthermore,
the construction of well-established hierarchical matrices ($H$-matrices) is based on the idea of low-rank approximation of matrix blocks
for matrices generated by kernel functions with point singularity, see
for example \cite{HaKh:Lin}, \cite{hackbusch2006low}. The low-rank approximation to kernel functions was also recently considered in  \cite{khan2025parametricTT,khan2025parametricHmatrix,li2025hierarchical}.
The other well-developed computational technique is based on the general principle
 of model reduction~\cite{BenFs:2024,benner2017model}. In the present paper, the term \emph{mesh-free} refers specifically to avoiding the construction and decomposition of a full fine-grid tensor in the computational domain. Depending on the available data, the method still uses function values at tensor-product Chebyshev nodes, or values reconstructed from structured grid-based or CP-formatted input.

This paper is partially motivated by recent advances in tensor techniques for 3D Chebyshev interpolation \cite{Beben:11,Kress:21,Chebfun:2014,Trefeth:2017,Kress:24}.
Chebyshev polynomials \cite{Cheb:1859} are among the most commonly used sets of interpolating functions due to many beneficial features \cite{BerTre,Ber:88,mason2002chebyshev}.
Error analysis for tensor product interpolation by Chebyshev polynomials was presented for example in \cite{HaKh:Lin}.
MATLAB implementations of the Chebyshev interpolation method up to 3D are available within the \texttt{chebfun} package~\cite{trefethen2019,Chebfun:2014,Trefeth:2017,townsend2013chebfun2,townsend2015chebfun2,townsend2014thesis,townsend2020ballfun}.
The basic version of the trivariate \texttt{chebfun3}~\cite{Trefeth:2017} is used to reduce the number of functional calls, and applies to ${\bf T} \in \R^{m_1\times m_2 \times m_3}$ represented in Tucker format, which contains the values of the input function at Chebyshev nodes, computed by using only \texttt{chebfun} for 1D polynomial interpolation.
Approximation of the target tensor ${\bf T}$ in Tucker format is performed via a variant of the heuristic method adaptive cross approximation (ACA) \cite{Beben:11} that combines ACA for the 2D slices with agglomeration of slices along the third spatial direction.
The number of functional calls in this approach is $\mathcal{O}(mr^2)$ and may be beneficial for low-rank smooth functions. Here and in the complexity discussion below, for notational simplicity we use $m := \max\{m_1,m_2,m_3\}$ and $r := \max\{r_1,r_2,r_3\}$ for the maximal polynomial degree and Tucker rank, respectively.
In contrast, the direct tensor product Chebyshev interpolant requires $m^3$ functional calls.
A modification of \texttt{chebfun3} that reduces the number of functional calls was presented in \cite{Kress:21}.

The powerful interpolatory methods face a significant challenge.
For functions with sufficient smoothness, this approach can achieve quasi-optimal accuracy with a relatively small basis set. However, for functions with low regularity or localized features, such as those with sharp gradients or singularities, achieving a satisfactory accuracy may require a very large number of interpolating functions (i.e., a high polynomial degree $m$),
which makes further use of such interpolants in scientific computing non-tractable. That is the main bottleneck for approximation of multivariate functions by Chebyshev polynomials in large scale scientific computing involving low regularity functions.

In this paper, we introduce and analyze a Chebyshev-Tucker tensor representation which combines a functional Tucker format in the Chebyshev polynomial basis with the Tucker decomposition of the tensor of Chebyshev coefficients.
This leads to a mesh-free Tucker representation of a multivariate function, either given explicitly in analytic form or available in the form of the CP/Tucker tensor format.
First, we calculate the 3D tensor of Chebyshev interpolation coefficients for the given multivariate function, thus obtaining the core tensor of the functional Tucker format. Then we compute the Tucker decomposition of this core tensor with a quasi-optimal $\varepsilon$-rank parameter. Thus we obtain the two-level Chebyshev-Tucker tensor representation with controllable accuracy, which, depending on application, can be used either in the initial functional form or easily converted into a Tucker tensor living on an arbitrarily large 3D grid.

Ultimately, this representation combines Chebyshev interpolation polynomials discretized on arbitrarily large 3D grids with the Tucker decomposition of the tensor of Chebyshev coefficients.\footnote{Using non-orthogonal interpolating functions leads to small modification of the computational scheme.}
It allows to avoid the expenses of the rank-structured approximation of  function-related tensors  defined on fine spatial grids in $\mathbb{R}^3$, while benefiting from the Tucker decomposition of the core tensor of Chebyshev coefficients, which leads to nearly optimal rank parameters as for the well established grid-based tensor methods.
The $\varepsilon$-ranks inherited from the Tucker-ALS decomposition are almost optimal and can be much smaller than the polynomial degree, $m$, of the initial Chebyshev interpolant which, in turn, is chosen independent of the grid sizes in grid-based methods.
Thus, we compute the nearly optimal Tucker decomposition of the 3D function without discretizing the function on the full grid in the computational domain, but only using its  values at Chebyshev nodes.
Note that in the case of functions with multiple cusps (or even stronger singularities), our method applies to the long-range component of the target function approximated in the range-separated (RS) tensor format \cite{benner2018range}.

We underline that, given accuracy $\varepsilon>0$, our main goal in this paper is to construct numerical methods that allow to compute the mesh-free Tucker tensor approximation of the target function with almost optimal $\varepsilon$-rank ${\bf r}=(r_1,r_2,r_3)$ which is usually much smaller than the Chebyshev polynomials degree $m$ that guarantees the accuracy $\varepsilon$ of the Chebyshev interpolant.
In turn, the parameter $m$ is supposed to be much smaller than the grid size $n$ that is required for the $\varepsilon$-accurate grid-based discretization of the input function (with the mesh-size $h$),
\begin{equation}
  |{\bf r}|\ll m \ll n = O(h^{-1}).
\end{equation}
The resulting Tucker rank is the most important parameter governing the complexity of subsequent multilinear algebra operations in respective applications.

We present numerical experiments demonstrating the efficiency of the proposed techniques on nontrivial examples, including the computation of multi-particle electrostatic potentials arising in bio-molecular modeling and lattice type structures calculations in material science.
In particular, it is gainfully applied to the long-range part of the singular many-particle electrostatic potentials in $\mathbb{R}^3$, obtained by the RS tensor format \cite{benner2018range}.
The low-rank tensor representations of the long-range part can be computed with controllable accuracy, at a cost that depends only weakly on the number of particles~\cite{benner2018range}, while the initial singular multi-particle electrostatic potential cannot be represented in the low-rank Tucker/CP formats.
In this RS setting, the input long-range tensor is represented in CP format with the CP-rank proportional to the number of particles, however, the resulting Tucker ranks after ChebTuck compression grow only logarithmically with the number of particles $N$ (see \cite[Theorem 3.1, Tables 1, 2]{benner2018range}, \cref{cor:cheb_tuck_error_cp_long_range} and \Cref{sec:logscaling}).
A more detailed discussion of the RS tensor format is beyond the scope of this paper.
Hence, the practical gain of the ChebTuck stage is to avoid the high-cost rank reduction of the CP input tensor living on the large spatial grid of size $n^{\otimes 3}$ in long-range multilinear computations and to replace direct dependence on the large univariate grid size $n$ in multilinear computations by dependence on a moderate Chebyshev degree $m$ and a compact Tucker core, which is significantly more attractive for subsequent algebraic operations.
In particular, it saves up to 98\% of the parameters in the ChebTuck representation of the long-range part of the multi-particle potential, compared to the CP representation of the same function on the full grid, while maintaining highly accurate approximations, see \Cref{tab:compression_ratio} in \Cref{sec:logscaling}.

In summary, we state the main benefits of the proposed method as follows:
\begin{itemize}
  \item {\bf Fast and accurate mesh-free decomposition:} The method yields a quasi-optimal mesh-free Tucker decomposition for both regular functions and the long-range part of functions with singularities~\cite{benner2018range}, such as the multi-particle electrostatic potential, generated by a weighted sum of 3D Newton kernels, by using a rather small number of Chebyshev polynomials.
  Notice that having at hand the low-rank representation of the long-range part of the singular multi-particle electrostatic potential, the interaction energy and forces of the system can be easily computed, see \cite{benner2018range,RSDock:23}.

  \item {\bf Efficiency for CP-formatted input:} In case of input functions presented in the canonical (CP) tensor format, our method needs only a small number of 1D Chebyshev interpolations of skeleton vectors, followed by a reduced HOSVD~\cite{khoromskij2009multigrid,khoromskaia2022ubiquitous} of the Chebyshev coefficient tensor, thus resulting in a complexity that scales linearly in the degree $m$ of the Chebyshev polynomials. A similar complexity scaling can be achieved in the case of Tucker-formatted input.
  \item {\bf Compatibility with grid-based methods:} By discretizing the Chebyshev polynomials in the final ChebTuck format, one obtains a standard low-rank Tucker tensor defined on any desired fine grid. This avoids the direct, and often prohibitively expensive, decomposition of a large grid-based tensor in $\mathbb{R}^{n \times n \times n}$, while reproducing the results of well-established grid-based tensor numerical methods~\cite{khoromskij2009multigrid,khoromskaia2014grid}.
\end{itemize}

The remainder of the paper is organized as follows.
In \Cref{sec:HybridChebTuck}, we introduce the Chebyshev-Tucker representation and collect relevant notations.
\Cref{sec:sect3} presents the numerical schemes for constructing the approximation of trivariate functions, either given explicitly in analytical form or only available on a regular grid, and provides corresponding error and complexity analysis.
\Cref{sec:app_multi_part_mod}  first outlines the well-developed techniques for grid-based low-rank tensor representations of singular multi-particle interaction potentials, see \cite{RSDock:23,benner2018range}. We then discuss in detail the Chebyshev-Tucker approximation of trivariate functions with singularities (on the example of the classical Newton kernel). We demonstrate the efficiency of our methods for  approximating the long-range part of the multi-particle electrostatic potentials of bio-molecular systems and lattice-type compounds of different sizes, and present various numerical tests illustrating the asymptotic performance of the proposed approach.
Finally, \Cref{sec:conclude} concludes the paper and outlines potential future research directions.

\section{The mesh-free hybrid Chebychev-Tucker tensor representation}\label{sec:HybridChebTuck}

In this section, we review the tensor notation needed throughout the paper and introduce the Chebyshev-Tucker tensor format for the approximation of multivariate functions.

Recall that a tensor of order $d$ is defined as a real multi-dimensional array over a $d$-tuple index set ${\bf A} = [a_{i_1, \cdots, i_d}] \equiv [a(i_1, \cdots, i_d)] \in \mathbb{R}^{n_1 \times \cdots \times n_d},$
with multi-index notation ${\bf i}= (i_1, \cdots, i_d)$, $i_{\ell} = 1, \cdots, n_{\ell}$.
To get rid of the exponential scaling in storage size, one can apply the rank-structured formatted separable approximations of tensors. For example, the \emph{$R$-term canonical (CP) tensor format} is defined by a finite sum of rank-1 tensors
\begin{equation}\label{eqn:canonical}
  {\bf U}_R =  {\sum}_{k =1}^{R} \xi_k {\bf u}_k^{(1)}  \otimes \cdots \otimes {\bf u}_k^{(d)},
  \quad \xi_k \in \mathbb{R}_+,
\end{equation}
where $R \in \mathbb{N}_+$ is the canonical rank, $\xi_k \in \mathbb{R}_+$ are positive weights,
and ${\bf u}_k^{(\ell)} \in \mathbb{R}^{n_{\ell}}$ are known as the canonical (skeleton) vectors.
The rank-${\bf r}$, ${\bf r}=(r_1,\cdots,r_d)$, \emph{orthogonal Tucker format} is specified as  the  set of tensors ${\bf V}\in \mathbb{R}^{n_1 \times \cdots \times n_d}$ parametrized as
\begin{equation}\label{eqn:Tucker}
  {\bf V} =  \sum_{\nu_1 =1}^{r_1} \cdots \sum_{\nu_d =1}^{r_d} \beta_{\nu_1, \cdots, \nu_d}
  {\bf v}_{\nu_1}^{(1)}  \otimes \cdots \otimes {\bf v}_{\nu_d}^{(d)}  \equiv
  \boldsymbol{\beta} \times_1 V^{(1)} \times_2 V^{(2)}\cdots \times_d V^{(d)},
\end{equation}
where $\{ {\bf v}_{\nu_{\ell}}^{(\ell)}\}_{\nu_{\ell} = 1}^{r_{\ell}} \in \mathbb{R}^{n_{\ell}}$ is a set of orthonormal vectors for $\ell = 1, \cdots, d$.
Here, $\times_{\ell}$ denotes the contraction along the $\ell$-th mode of the Tucker core $\boldsymbol{\beta} \in \mathbb{R}^{r_1 \times \cdots \times r_d}$ with the second mode of the orthogonal matrices $V^{(\ell)} = [{\bf v}_1^{(\ell)}, \cdots, {\bf v}_{r_{\ell}}^{(\ell)}] \in \mathbb{R}^{n_{\ell}  \times r_{\ell}}$. Here and after we consider $d=3$ to simplify the presentation since (1) the extension to higher dimensions is straightforward, and (2) Tucker is not the recommended format for high dimensions due to the exponential scaling of the storage of the core tensor.

Furthermore, let $T_k(x) := \cos((k-1) \arccos(x))$ denote the $(k-1)$-st Chebyshev polynomial of the first kind, which forms a sequence of orthogonal polynomials on $[-1,1]$ with respect to the weight $w(x) = 1/\sqrt{1-x^2}$.
We use these polynomials as a fixed tensor-product basis and compress the corresponding coefficient tensor in Tucker format.
We call the resulting representation the Chebyshev-Tucker format, shortly ChebTuck format.
\begin{definition}[ChebTuck format]\label{def:ChebTuck}
  Let $T_{1:k}(x) = (T_1(x), \cdots, T_{k}(x))$ denote the vector-valued function of the first $k$ Chebyshev polynomials. We say that the real-valued function $f(x_1,x_2,x_3)$ defined on $[-1,1]^3$ is represented in the {\bf ChebTuck-$(\bf{m}, \bf{r})$ format} if it allows the factorization
  \begin{equation}\label{eq:cheb_tuck_def}
    f(x_1, x_2, x_3) = {\bf C} \times_1 T_{1:m_1}(x_1) \times_2 T_{1:m_2}(x_2) \times_3 T_{1:m_3}(x_3),
  \end{equation}
  where ${\bf C} \in \mathbb{R}^{m_1 \times m_2 \times m_3}$ is a rank-$\bf{r}$ Tucker tensor given by
  \begin{equation}\label{eq:cheb_tuck_Ccoef}
    {\bf C} = \boldsymbol{\beta} \times_1 V^{(1)} \times_2 V^{(2)} \times_3 V^{(3)}, \quad \boldsymbol{\beta} \in \mathbb{R}^{r_1 \times r_2 \times r_3}, \quad V^{(\ell)} \in \mathbb{R}^{m_{\ell} \times r_{\ell}}.
  \end{equation}
\end{definition}

Equivalently, the ChebTuck representation may be written in functional Tucker form as
\begin{equation}\label{eqn:cheb_func_tuck}
  f(x_1,x_2,x_3) = \boldsymbol{\beta} \times_1 v^{(1)}(x_1) \times_2 v^{(2)}(x_2) \times_3 v^{(3)}(x_3),
\end{equation}
where $v^{(\ell)}(x_\ell) = T_{1:m_\ell}(x_\ell) V^{(\ell)} \in \R^{r_\ell}$, that is, each entry of $v^{(\ell)}$ is a linear combination of Chebyshev polynomials. Hence, by discretizing the basis functions on any target grid, the final ChebTuck approximation can be converted to a standard algebraic Tucker tensor on that grid. If needed, this Tucker tensor can then be transformed further to canonical form by the standard Tuck-to-can procedure; see \cite[Def. 2.3, (2.6)]{khoromskij2009multigrid}, \cite[Rem. 2.7]{khoromskij2007low}, and \cite{Khor_bookQC_2018}.

We can see that the ChebTuck format is a degree ${\bf m}=(m_1,m_2,m_3)$ polynomial in the Chebyshev basis, where the coefficients are given as a Tucker tensor. For this reason, we sometimes also write a subscript ${\bf m}$ to denote the ChebTuck format of certain degrees $f_{\bf m}$.

\section{Numerical methods for hybrid Chebyshev-Tucker approximation of trivariate functions}\label{sec:sect3}

In what follows, we introduce the numerical methods for rank-structured approximation of trivariate functions in the ChebTuck format introduced in Definition \ref{def:ChebTuck}. The basic idea is to first compute the Chebyshev coefficient tensor ${\bf C}$, either in full tensor format or in some low-rank format, and then convert ${\bf C}$ to the desired Tucker format with quasi-optimal $\varepsilon$-rank. Our techniques can be applied directly to multivariate functions with moderate regularity for which the Chebyshev tensor product interpolation demonstrates satisfactory convergence rates. Furthermore, we extend the presented approach beyond the class of regular functions and successfully apply it to functions with multiple singularities by using the range-separated tensor decomposition.

\subsection{Classical multivariate Chebyshev interpolant}
\label{sec:ChebTuck}
\subsubsection{Functional case}
Given $f: [-1,1]^3 \to \mathbb{R}$, we consider its approximation by the tensor product of Chebyshev polynomials in the form (c.f. \cref{eq:cheb_tuck_def})
\begin{equation}\label{eq:3d_cheb}
  f(x_1,x_2,x_3) \approx \tilde{f}_{\bf m}(x_1,x_2,x_3) = \sum_{i_1=1}^{m_1} \sum_{i_2=1}^{m_2} \sum_{i_3=1}^{m_3} {\bf C}_{i_1,i_2,i_3} T_{i_1}(x_1) T_{i_2}(x_2) T_{i_3}(x_3),
\end{equation}
where $T_k(x)$ is the $(k-1)$-st Chebyshev polynomial, ${\bf C} \in \R^{m_1 \times m_2 \times m_3}$ is the Chebyshev coefficient tensor (CCT).

There are mainly two ways to interpret the approximant $\tilde{f}_{\bf m}$, which lead to two different methods to compute the CCT ${\bf C}$. The first way is to take $\tilde{f}_{\bf m}$ as a \emph{truncation} of the tensor product Chebyshev series expansion of $f$ and compute the coefficients by projection and integration. In this case, the Chebyshev coefficients are in fact straightforward to derive and given by the $w$-weighted inner product of $f$ and the corresponding Chebyshev basis, i.e., ${\bf C}_{i_1,i_2,i_3} = \int_{[-1,1]^3} f(\mathbf{x}) T_{i_1}(x_1) T_{i_2}(x_2) T_{i_3}(x_3)w(\mathbf{x}) \dd \mathbf{x}$, where $w(\mathbf{x}) = \prod_{\ell=1}^3 (1-x_\ell^2)^{-1/2}$ is the Chebyshev weight function,
see e.g., \cite{mason1980near,mason1982minimal}, but the computation is troublesome since one needs to evaluate those 3D integrals. Here, we used the notation $\mathbf{x} = (x_1,x_2,x_3)$.
The second way is to take $\tilde{f}_{\bf m}$ as an \emph{interpolation} of $f$ at Chebyshev nodes, namely, $\tilde{f}_{\bf m}$ satisfies the interpolatory property $\tilde{f}_{\bf m}(s_{i_1}^{(1)}, s_{i_2}^{(2)}, s_{i_3}^{(3)}) = f(s_{i_1}^{(1)}, s_{i_2}^{(2)}, s_{i_3}^{(3)})$ for all $i_\ell = 1, \ldots, m_\ell$, $\ell = 1,2,3$, where $s_{i_\ell}^{(\ell)}$ are the Chebyshev nodes of the second kind on each mode given by $s_{i_\ell}^{(\ell)} = \cos\left((i_\ell-1)\pi/(m_\ell -1)\right)$.
The derivation of the CCT ${\bf C}$ in this case is more involved while it turns out that the computation can be done efficiently by the discrete cosine transform (DCT). The accuracy of these two approaches is in practice similar \cite{mason1980near} and we will therefore focus on the interpolation approach in this paper.

For this purpose, we first define the function evaluation tensor ${\bf T} \in \R^{m_1 \times m_2 \times m_3}$ of $f$ at Chebyshev nodes as
\begin{equation}\label{eq:3d_eval_tensor}
  {\bf T}_{i_1,i_2,i_3} = f(s_{i_1}^{(1)}, s_{i_2}^{(2)}, s_{i_3}^{(3)}), \quad i_\ell = 1, \ldots, m_\ell, \quad \ell = 1,2,3
\end{equation}
and the inverse DCT matrix $W^{(\ell)} \in \R^{m_\ell \times m_\ell}$ as
\begin{equation}\label{eq:dct_matrix}
  W^{(\ell)} = \dfrac{2}{m_\ell-1}\begin{pmatrix}
    \frac{1}{4} T_1(s_1)        & \frac{1}{2} T_1(s_2)        & \frac{1}{2} T_1(s_3)        & \dots  & \frac{1}{4} T_{1}(s_{m_\ell})      \\
    \frac{1}{2} T_2(s_1)        & T_2(s_2)                    & T_2(s_3)                    & \dots  & \frac{1}{2} T_{2}(s_{m_\ell})      \\
    \vdots                      & \vdots                      & \vdots                      & \ddots & \vdots                             \\
    \frac{1}{4} T_{m_\ell}(s_1) & \frac{1}{2} T_{m_\ell}(s_2) & \frac{1}{2} T_{m_\ell}(s_3) & \dots  & \frac{1}{4} T_{m_\ell}(s_{m_\ell})
  \end{pmatrix},
\end{equation}
then by \cite[Equations 6.27---6.28]{mason2002chebyshev}, the CCT ${\bf C}$ can be computed by the DCT of the function evaluation tensor ${\bf T}$ as
\begin{equation}\label{eq:3d_cheb_coeff}
  {\bf C} = {\bf T} \times_1 W^{(1)} \times_2 W^{(2)} \times_3 W^{(3)}.
\end{equation}

Given an error tolerance $\varepsilon > 0$, we can approximate the coefficient tensor ${\bf C}$ by a Tucker tensor with almost optimal rank parameter ${\bf r}=(r_1,r_2,r_3)$ using the Tucker-ALS algorithm denoted by $\texttt{tuck\_als}({\bf C}, \varepsilon)$
\begin{equation}\label{eq:tucker_of_cheb_coeff}
  {\bf C} \approx \hat{\bf C} = \boldsymbol{\beta} \times_1 V^{(1)} \times_2 V^{(2)} \times_3 V^{(3)}, \quad [\boldsymbol{\beta}, V^{(1)}, V^{(2)}, V^{(3)}] = \texttt{tuck\_als}({\bf C}, \varepsilon)
\end{equation}
with $\boldsymbol{\beta} \in \R^{r_1 \times r_2 \times r_3}$ and $V^{(\ell)} \in \R^{m_\ell \times r_\ell}$, $\ell = 1,2,3$.
See \cref{rem:tucker_alg} for the discussions on Tucker-ALS and other Tucker decomposition algorithms.
Substituting ${\bf C}$ in \cref{eq:3d_cheb} by its Tucker format \cref{eq:tucker_of_cheb_coeff} gives the ChebTuck approximation of $f$, denoted by $\hat{f}_{\bf m}$. By \cref{eqn:cheb_func_tuck}, the ChebTuck approximation of $f$ can be written as
\begin{equation}\label{eq:cheb_tuck_func}
  f(\textbf{x}) \approx \tilde{f}_{\bf m}(x_1,x_2,x_3) \approx \hat{f}_{\bf m}(x_1,x_2,x_3) = \boldsymbol{\beta} \times_1 v^{(1)}(x_1) \times_2 v^{(2)}(x_2) \times_3 v^{(3)}(x_3).
\end{equation}
where $v^{(\ell)}_{j_\ell}(x_\ell) = \sum_{i_\ell = 1}^{m_\ell} V^{(\ell)}_{i_\ell,j_\ell} T_{i_\ell}(x_\ell)$ for $j_\ell = 1, \ldots, r_\ell$, $\ell = 1,2,3$.

In summary, this leads to the following straightforward \cref{alg:cheb_tuck_func} for constructing the Chebyshev-Tucker approximation of a function $f: [-1,1]^3 \to \mathbb{R}$.

\begin{algorithm}[tbh]
  \caption{ChebTuck Approximation: functional case}\label{alg:cheb_tuck_func}
  \begin{algorithmic}[1]
    \Require Given a function $f: [-1,1]^3 \to \mathbb{R}$, the Chebyshev degrees ${\bf m} = (m_1, m_2, m_3)$ and \texttt{tuck\_als} tolerance $\varepsilon > 0$.
    \Ensure The ChebTuck format $\hat{f}_{\bf m}$ of $f$.
    \State Compute the function evaluation tensor ${\bf T}$ by \cref{eq:3d_eval_tensor}, then the CCT ${\bf C}$ by \cref{eq:3d_cheb_coeff}. \label{line:cheb_coeff}
    \State Compute the Tucker decomposition: $[\boldsymbol{\beta}, V^{(1)}, V^{(2)}, V^{(3)}] = \texttt{tuck\_als}({\bf C}, \varepsilon)$.\label{line:tucker_decomp}
    \State Form the factor functions $v^{(\ell)}_{j_\ell}(x_\ell) = \sum_{i_\ell = 1}^{m_\ell} V^{(\ell)}_{i_\ell,j_\ell} T_{i_\ell}(x_\ell)$ for $\ell = 1,2,3$.\label{line:factor_functions}
    \State Return $\hat{f}_{\bf m}(x_1,x_2,x_3) = \boldsymbol{\beta} \times_1 v^{(1)}(x_1) \times_2 v^{(2)}(x_2) \times_3 v^{(3)}(x_3)$. \label{line:return_cheb_tuck}
  \end{algorithmic}
\end{algorithm}
\begin{remark}\label{rem:tucker_alg}
  Notice that $\texttt{tuck\_als}$ in line~\ref{line:tucker_decomp} of \cref{alg:cheb_tuck_func} implementing the classical
  Tucker-ALS algorithm \cite{de2000multilinear,de2000best} for this purpose, is widely used in practice for moderate $d$ and $m$. It enjoys strong theoretical guarantees such that for any given $\varepsilon > 0$, it approximates the original tensor in the sense of relative error $\|{\bf C} - \hat{\bf C}\|_{\mathrm{F}} \leq \varepsilon\|{\bf C}\|_{\mathrm{F}}$ with almost optimal rank parameters, where $\| \cdot \|_{\mathrm{F}}$ is the Frobenius norm of a tensor. However, it is computational expensive since it requires the access to all entries of the tensor ${\bf C}$, amounting to $O(m^3)$ space complexity and successive SVDs of the unfolding matrices, amounting to $O(m^4)$ algebraic operations.
  Alternatives to the classical Tucker-ALS are, e.g., the multigrid algorithm \cite{khoromskij2009multigrid}, cross-based algorithms \cite{oseledets2008tucker,DS:2014,Kress:21,rakhuba2015fast,Trefeth:2017}, or randomized algorithms \cite{che2019randomized,kressner2023streaming}. In case of the relatively small rank of the input tensor, those algorithms are more efficient in terms of computational cost and memory usage. In particular, using the cross-based algorithms can avoid the evaluation of the full tensor ${\bf T}$. However, they contain several heuristics, especially the cross-based algorithms have no theoretical accuracy guarantees for given $\varepsilon$. The discussions on the difference and choice of Tucker decomposition algorithms are beyond the scope of this paper. We use the classical Tucker-ALS algorithm in our theoretical analysis and numerical experiments demonstrated reliable
  efficiency for low-rank approximation of the coefficient tensor ${\bf C}$.
\end{remark}

\subsubsection{Algebraic case with full tensor input}

In many applications, see \cite{Khor-book_2018,Khor_bookQC_2018,RSDock:23} among others, including ours on multi-particle modeling which will be discussed in \Cref{sec:app_multi_part_mod}, the function $f$ may not be given explicitly in the full domain $[-1,1]^3$,  but only the data on the equi-spaced grid are available.
Specifically, we only have access to the function values at uniform grid points collected in a tensor ${\bf F} \in \mathbb{R}^{n_1 \times n_2 \times n_3}$, ${\bf F}_{i_1,i_2,i_3} = f(t_{i_1}^{(1)}, t_{i_2}^{(2)}, t_{i_3}^{(3)})$ for $i_\ell = 1, \ldots, n_\ell$, $\ell = 1,2,3$. Here $t_{i_\ell}^{(\ell)} = -1 + (i_\ell-1) h_\ell$ is the grid point and $h_\ell = 2/n_\ell$ is the grid size.

In this case, we first extrapolate the function to the full domain $[-1,1]^3$ by constructing a multivariate cubic spline interpolation $Q(x_1,x_2,x_3)$ of the data ${\bf F}$ on the uniform grid, i.e., $Q(x_1,x_2,x_3)$ is a $C^2$ piecewise cubic polynomial satisfying the interpolatory property $Q(t_{i_1}^{(1)}, t_{i_2}^{(2)}, t_{i_3}^{(3)}) = {\bf F}_{i_1,i_2,i_3}$, and then call \cref{alg:cheb_tuck_func} with $f = Q$. This procedure is summarized in \cref{alg:cheb_tuck_alg}. Note that here and forehand, we assume $n_\ell \geq 4$ naturally to ensure the existence of the cubic spline.

\begin{algorithm}[tbh]
  \caption{ChebTuck Approximation: algebraic case}\label{alg:cheb_tuck_alg}
  \begin{algorithmic}[1]
    \Require Given a tensor ${\bf F}\in \mathbb{R}^{n_1 \times n_2 \times n_3}$ and $\mathbf{m}$, $\varepsilon > 0$ as in \cref{alg:cheb_tuck_func}.
    \Ensure The ChebTuck format $\hat{f}_{\bf m}$ of $f$.
    \State Construct the multivariate cubic spline interpolation $Q(x_1,x_2,x_3)$ of the data ${\bf F}$ on the uniform grid by, e.g., \texttt{interpn} in MATLAB.
    \State Call \cref{alg:cheb_tuck_func} with $f = Q$.
  \end{algorithmic}
\end{algorithm}

\subsubsection{Algebraic case with CP tensor input}\label{ssec:alg_cp_input}
In fact, the multivariate cubic spline interpolation, see e.g. \cite[Chapter 17]{de1978practical}, is computationally expensive. However, in many applications of interest, e.g., quantum chemistry \cite{Khor_bookQC_2018}, multi-particle modeling \cite{Khor-book_2018} and the discretization of certain differential operators \cite{kazeev2013low}, the target input tensor ${\bf F}$ has some inherent low-rank structure and thus can be cheaply represented in some low-rank tensor format. In such cases, the computation of the evaluation tensor ${\bf T}$ can be done more efficiently. In particular, we consider here the case where ${\bf F}$ is given in the CP format, i.e.,
\begin{equation}\label{eq:func_val_cp}
  {\bf F} = \sum_{k=1}^{R} \xi_k {\bf a}_k^{(1)} \otimes {\bf a}_k^{(2)} \otimes {\bf a}_k^{(3)},
\end{equation}
where ${\bf a}_k^{(\ell)} \in \R^{n_\ell}$ are the canonical vectors of each mode. In this case, we only need to construct \emph{univariate} cubic splines of the canonical vectors ${\bf a}_k^{(\ell)}$ in each mode, and make tensor products and summations of them. Specifically, let $q_k^{(\ell)}(x_\ell)$ be the cubic spline of the data ${\bf a}_k^{(\ell)}$ on the uniform grid $\{ t_{i_\ell}^{(\ell)} \}_{i_\ell=1}^{n_\ell}$, i.e., $q_k^{(\ell)}(t_{i_\ell}^{(\ell)}) = {\bf a}_k^{(\ell)}(i_\ell)$ and then construct the sum of the tensor product of the univariate splines as
\begin{equation}\label{eq:cp_2_spline}
  Q(x_1,x_2,x_3) = {\sum}_{k=1}^{R}\xi_k q_k^{(1)}(x_1) q_k^{(2)}(x_2) q_k^{(3)}(x_3).
\end{equation}
Calling \cref{alg:cheb_tuck_func} with $f = Q$, one sees immediately from the definitions of the CCT ${\bf C}$ \cref{eq:3d_eval_tensor,eq:3d_cheb_coeff} that ${\bf C}$ computed in line~\ref{line:cheb_coeff} is also in CP format. Specifically, let ${\bf q}_k^{(\ell)} \in \R^{m_\ell}$ contain the function evaluations of the univariate spline $q_k^{(\ell)}(x)$ at $\{ s_{i_\ell}^{(\ell)} \}_{i_\ell=1}^{m_\ell}$, i.e., ${\bf q}_k^{(\ell)}(i_\ell) = q_k^{(\ell)}(s_{i_\ell}^{(\ell)})$, then
by definition of the function evaluation tensor ${\bf T}$ in \cref{eq:3d_eval_tensor} and the CCT ${\bf C}$ in \cref{eq:3d_cheb_coeff}, we have
\begin{equation}\label{eq:cheb_coeff_cp}
  {\bf T} = {\sum}_{k=1}^{R} \xi_k {\bf q}_k^{(1)} \otimes {\bf q}_k^{(2)} \otimes {\bf q}_k^{(3)} \implies {\bf C} = {\sum}_{k=1}^{R} \xi_k {\bf c}_k^{(1)} \otimes {\bf c}_k^{(2)} \otimes {\bf c}_k^{(3)},
\end{equation}
with ${\bf c}_k^{(\ell)} = W^{(\ell)} {\bf q}_k^{(\ell)}$ for $\ell = 1,2,3$ and $W^{(\ell)}$ being the inverse DCT matrix defined in \cref{eq:dct_matrix}.
From here, we can leverage the reduced HOSVD (\texttt{rhosvd}) \cite{khoromskij2009multigrid} method instead of full \texttt{tuck\_als} to compress ${\bf C}$ directly from canonical to Tucker format, without forming the full tensor ${\bf C}$. The rest of the algorithm remains the same as Algorithm \ref{alg:cheb_tuck_func}. This procedure is summarized in Algorithm \ref{alg:cheb_tuck_alg_cp}.
\begin{algorithm}[tbh]
  \caption{ChebTuck Approximation: algebraic case (CP tensor input)}
  \label{alg:cheb_tuck_alg_cp}
  \begin{algorithmic}[1]
    \Require Given a tensor ${\bf F}\in \mathbb{R}^{n_1 \times n_2 \times n_3}$ in CP format, the Chebyshev degrees ${\bf m} = (m_1, m_2, m_3)$ and \texttt{rhosvd} tolerance $\varepsilon > 0$.
    \Ensure The ChebTuck format $\hat{f}_{\bf m}$ of $f$.
    \State Compute the splines $q_k^{(\ell)}(x_\ell)$ and the evaluations ${\bf q}_k^{(\ell)}$ for $k = 1, \ldots, R$, $\ell = 1,2,3$.\label{line:spline_eval}
    \State Compute the canonical vectors ${\bf c}_k^{(\ell)} = W^{(\ell)} {\bf q}_k^{(\ell)}$ of ${\bf C}$ for $k = 1, \ldots, R$, $\ell = 1,2,3$.\label{line:cheb_coeff_cp}
    \State Compute the reduced Tucker decomposition: $[\boldsymbol{\beta}, V^{(1)}, V^{(2)}, V^{(3)}] = \texttt{rhosvd}({\bf C}, \varepsilon)$
    \State Perform lines \ref{line:factor_functions}-\ref{line:return_cheb_tuck} in Algorithm \ref{alg:cheb_tuck_func}.
  \end{algorithmic}
\end{algorithm}

\subsection{Error bounds for the Chebyshev-Tucker approximation}\label{sec:func_cp_cheb}
In this section, we always assume $m = m_1 = m_2 = m_3$ without loss of generality to simplify the notation. While we state the results for general tensor order $d$ since they also hold for $d \geq 3$, the proofs specialize to $d = 3$ for clarity of presentation.

\subsubsection{Functional case}
We first consider the ChebTuck approximation $\hat{f}_{\bf m}$ of a function $f$ obtained by Algorithm \ref{alg:cheb_tuck_func}, i.e. the functional case. It is well-known that multivariate Chebyshev interpolation $\tilde{f}_{\bf m}$ converges exponentially fast to the function $f$ in the uniform norm over $[-1,1]^3$ with respect to the degree ${\bf m}$ of the Chebyshev polynomials, provided that $f$ has sufficient regularity. Our ChebTuck format introduces an additional error due to the Tucker approximation of the Chebyshev coefficient tensor ${\bf C}$, which is controlled by the Tucker-ALS truncation error $\varepsilon$. In the following proposition, we show that the ChebTuck approximation $\hat{f}_{\bf m}$ is almost as accurate as the full Chebyshev interpolation $\tilde{f}_{\bf m}$ provided that the Tucker approximation error $\varepsilon$ is sufficiently small.

\begin{proposition}\label{prop:cheb_tuck_error}
  Given notations developed above, we have
  \begin{equation}
    \|f - \hat{f}_{\bf m}\|_{\infty} \leq \| f - \tilde{f}_{\bf m} \|_{\infty} + m^{d/2} \varepsilon \| \bf{C} \|_{\mathrm{F}},
  \end{equation}
  where $\| \cdot \|_{\infty}$ is the uniform norm over $[-1,1]^3$.
\end{proposition}
\begin{proof}
  By triangle inequality, it suffices to prove $\| \hat{f}_{\bf m} - \tilde{f}_{\bf m} \|_{\infty} \leq m^{3/2} \varepsilon \| \bf{C} \|_{\mathrm{F}}$.
  Indeed, by the definition of the ChebTuck format \cref{eq:cheb_tuck_func} and the fact that Chebyshev polynomials are bounded by 1, we have
  \begin{equation*}
    \begin{aligned}
      \| \tilde{f}_{\bf m} - \hat{f}_{\bf m} \|_{\infty} & = \left\| \sum_{i_1=1}^{m} \sum_{i_2=1}^{m} \sum_{i_3=1}^{m} ({\bf C}_{i_1,i_2,i_3} - \hat{\bf C}_{i_1,i_2,i_3}) T_{i_1}(x_1) T_{i_2}(x_2) T_{i_3}(x_3) \right\|_{\infty}                                               \\
                                                         & \leq \sum_{i_1=1}^{m} \sum_{i_2=1}^{m} \sum_{i_3=1}^{m} |{\bf C}_{i_1,i_2,i_3} - \hat{\bf C}_{i_1,i_2,i_3}| \leq m^{3/2} \| {\bf C} - \hat{\bf{C}} \|_{\mathrm{F}} \leq m^{3/2} \varepsilon \| {\bf C} \|_{\mathrm{F}},
    \end{aligned}
  \end{equation*}
  where we used the Cauchy-Schwarz inequality and the relative truncation error $\varepsilon$ of the Tucker-ALS approximation in the last two inequalities.
\end{proof}
\subsubsection{Algebraic case}\label{ssec:alg_cp_error}
Now, we consider the ChebTuck approximation in the algebraic case, i.e., given a tensor ${\bf F} \in \R^{n_1 \times n_2 \times n_3}$ assumed to discretize a black-box function $f$ on the uniform grid.
In this case, the metric for the error of the approximation is the difference between the ChebTuck approximation $\hat{f}_{\bf m}$ evaluated at the uniform grid and the tensor ${\bf F}$, i.e., we define
\begin{equation}\label{eq:error_metric}
  \err(\hat{f}_{\bf m}, {\bf F}) := \max_{i_1,i_2,i_3} \left| \hat{f}_{\bf m}(t_{i_1}^{(1)}, t_{i_2}^{(2)}, t_{i_3}^{(3)}) - {\bf F}_{i_1,i_2,i_3} \right|.
\end{equation}
In other words, we hope to recover the original input ${\bf F}$ when we evaluate the ChebTuck format $\hat{f}_{\bf m}$ at the uniform grid. Since we know nothing about the function $f$ except for the values on the grid, for \cref{alg:cheb_tuck_alg}, we can say nothing about the error of the Chebyshev interpolation of the multivariate cubic spline $Q$, thus neither the ChebTuck approximation $\hat{f}_{\bf m}$. For example, for highly non-regular function $f$
we would not expect the Chebyshev interpolation thus also the ChebTuck approximation to be accurate.

\begin{remark}\label{rem:algebr-apprther}
In this paper, we assume that ChebTuck format serves for approximating regular functions which admit accurate
Chebyshev interpolation with a moderate polynomial degree $m$.
 Then the total error of the ChebFun approximation is basically determined by two main ingredients, that is
 the error of the cubic spline interpolation over the grid data given on a uniform mesh and the accuracy of the Chebyshev polynomial
 interpolation based on slightly perturbed functional values
 at the Chebyshev interpolation nodes. The first point is an issue of classical approximation theory of spline interpolation of functions of several variables and there are standard error estimates in terms of mesh-size $h$ and bounds on the respective derivatives of target function given on a grid (see Tables 1 - 3 demonstrating the convergence rate for both regular and non-regular functions). Concerning the second issue, there is the classical approximation theory via polynomials for different classes of regular functions.

 Both issues are beyond the scope of this paper and will be considered  in all details for special applications elsewhere.
\end{remark}

Fortunately, in the case of CP tensor input, we can provide a bound for the error of the ChebTuck approximation $\hat{f}_{\bf m}$ depending on the errors of the \emph{univariate} Chebyshev interpolations of the canonical vectors ${\bf a}_k^{(\ell)}$ in each mode. Note that the canonical vectors ${\bf c}_k^{(\ell)}$ of the Chebyshev coefficient tensor ${\bf C}$ are nothing but the \emph{univariate} Chebyshev coefficients of the splines $q_k^{(\ell)}(x)$. In other words, let $g_k^{(\ell)}(x) := \sum_{j=1}^{m} {\bf c}_k^{(\ell)}(j) T_{j}(x)$, then $g_k^{(\ell)}(x)$ is the Chebyshev interpolation of $q_k^{(\ell)}(x)$. Let $\hat{\bf C}$ be the RHOSVD truncation of ${\bf C}$, then the ChebTuck approximation error is given by
  \begin{equation}\label{eq:func_cp_error_triangle}
    \err(\hat{f}_{\bf m}, {\bf F}) \leq \err(\tilde{f}_{\bf m}, {\bf F}) + \| \hat{f}_{\bf m} - \tilde{f}_{\bf m} \|_{\infty}
  \end{equation}
where $\tilde{f}_{\bf m}$ is the Chebyshev interpolant \cref{eq:3d_cheb} with coefficients ${\bf C}$ given in \cref{eq:cheb_coeff_cp}, i.e., before RHOSVD truncation. $\tilde{f}_{\bf m}$ is thus given as
\begin{equation}\label{eq:ftilde_cp}
  \tilde{f}_{\bf m}(x_1,x_2,x_3) = \sum_{k=1}^{R} \xi_k g_k^{(1)}(x_1) g_k^{(2)}(x_2) g_k^{(3)}(x_3).
\end{equation}
Then we have the following bound for the first term in the right-hand side of \cref{eq:func_cp_error_triangle}.
\begin{lemma}\label{lem:ftilde-F}
  Let $\tilde{f}_{\bf m}$ be of the form \cref{eq:3d_cheb} with coefficients ${\bf C}$ in \cref{eq:cheb_coeff_cp}, i.e., $\tilde{f}_{\bf m}$ takes the form \cref{eq:ftilde_cp}.
  We further assume that the canonical vectors ${\bf a}_k^{(\ell)}$ of the CP tensor ${\bf F}$ in \cref{eq:func_val_cp} are at least $\ell_\infty$ normalized, i.e., $\| {\bf a}_k^{(\ell)} \|_{\infty} \leq 1$.
  Let $\delta_k^{(\ell)} := \| q_k^{(\ell)} - g_k^{(\ell)} \|_{\infty}$ be the error of the univariate Chebyshev interpolation of $q_k^{(\ell)}(x)$ and $\delta := \max_{k,\ell} \delta_k^{(\ell)}$, then we have
  \begin{equation}\label{eq:func_cp_error}
    \err(\tilde{f}_{\bf m}, {\bf F}) \leq \| \boldsymbol{\xi} \|_1 \left( (1 + \delta)^d - 1 \right),
  \end{equation}
  where $\| \boldsymbol{\xi} \|_1 = \sum_{k=1}^{R} \xi_k$ is the $\ell_1$ norm of the CP weights.
\end{lemma}
\begin{proof}
  Recall that by construction, $q_k^{(\ell)}$ interpolates the data ${\bf a}_k^{(\ell)}$ on the uniform grid $\{ t_{i_\ell}^{(\ell)} \}_{i_\ell=1}^{n_\ell}$ and the Chebyshev interpolants have uniform error $\delta_k^{(\ell)}$. Therefore, for each $\ell = 1, \ldots, d$, $k = 1, \ldots, R$ and $i_\ell = 1, \ldots, n_\ell$, we have
  \begin{equation}
    |{\bf a}_k^{(\ell)}(i_\ell) - g_k^{(\ell)}(t_{i_\ell}^{(\ell)})| = |q_k^{(\ell)}(t_{i_\ell}^{(\ell)}) - g_k^{(\ell)}(t_{i_\ell}^{(\ell)})| \leq \| q_k^{(\ell)} - g_k^{(\ell)} \|_{\infty} = \delta_k^{(\ell)}
  \end{equation}
  and in particular,
  \begin{equation}
    |g_k^{(\ell)}(t_{i_\ell}^{(\ell)})| \leq |{\bf a}_k^{(\ell)}(i_\ell)| + \delta_k^{(\ell)} \leq \| {\bf a}_k^{(\ell)} \|_{\infty} + \delta_k^{(\ell)} \leq 1 + \delta_k^{(\ell)},
  \end{equation}
  where we have used the assumption that the canonical vectors ${\bf a}_k^{(\ell)}$ are at least $\ell_\infty$ normalized.
  Therefore, for any $i_1, \ldots, i_d$, by the definition of $\tilde{f}_{\bf m}$ in \cref{eq:ftilde_cp} and ${\bf F}$ in \cref{eq:func_val_cp}, we have
  \begin{equation*}
    \begin{aligned}
           & \left|\tilde{f}_{\bf m}(t_{i_1}^{(1)}, \ldots, t_{i_d}^{(d)}) - {\bf F}_{i_1, \ldots, i_d}\right| \leq  {\sum}_{k=1}^{R} \xi_k \left|\prod_{\ell=1}^{d} g_k^{(\ell)}(t_{i_\ell}^{(\ell)}) - \prod_{\ell=1}^{d} {\bf a}_k^{(\ell)}(i_\ell) \right| \\
      \leq & \sum_{k=1}^{R} \xi_k \sum_{\ell=1}^{d} \left| \prod_{r=1}^{\ell-1} g_k^{(r)}(t_{i_\ell}^{(r)}) \right| \left|g_k^{(\ell)}(t_{i_\ell}^{(\ell)}) - {\bf a}_k^{(\ell)}(i_\ell) \right| \left| \prod_{r=\ell+1}^d  {\bf a}_k^{(r)}(i_r) \right|                         \\
      \leq & \sum_{k=1}^{R} \xi_k \sum_{\ell=1}^{d} \prod_{r=1}^{\ell - 1} \left( 1 + \delta_k^{(r)} \right) \delta_k^{(\ell)} \leq \sum_{k=1}^{R} \xi_k \sum_{\ell=1}^{d} \left( 1 + \delta \right)^{\ell-1} \delta = \| \boldsymbol{\xi} \|_1 \left( (1+\delta)^d - 1 \right),
    \end{aligned}
  \end{equation*}
  where the second inequality results from applying a standard telescopic argument.
\end{proof}
If, in addition, the Chebyshev interpolation error satisfies $\delta < 1/d$, then we can further improve the bound \cref{eq:func_cp_error} to a linear growth with respect to the dimension $d$, as stated in the following corollary.
\begin{corollary}\label{cor:ftilde-F}
  Under the assumptions of \cref{lem:ftilde-F}, further assuming $\delta d < 1$ yields
  \begin{equation}
    \err(\tilde{f}_{\bf m}, {\bf F}) \leq \| \boldsymbol{\xi} \|_1 \min\left\{\dfrac{\delta d}{1 - \delta d}, (\e - 1) \delta d\right\}.
  \end{equation}
\end{corollary}
\begin{proof}
  First note that $(1+\delta)^d \leq \e^{\delta d}$. Then the desired result follows from the elementary inequalities $\e^x\leq (\e - 1) x + 1 \text{ and } \e^x \leq (1-x)^{-1}$ for $x \in (0,1)$.
\end{proof}
The following lemma states that the error of the univariate Chebyshev interpolation of the cubic spline $q_k^{(\ell)}$ decays cubically with respect to the Chebyshev degree $m$.
\begin{lemma}\label{lem:cheb_spline_error}
  Let $V_k^{\ell} < +\infty $ be the bounded variation of the third derivative of the spline $q_k^{(\ell)}$ and $V = \max_{k,\ell} V_k^{\ell}$. Then we have
  \begin{equation}
    \delta_k^{(\ell)} \leq \dfrac{4 V_k^{\ell}}{3\pi (m - 3)^3} \text{ for } k = 1, \ldots, R,\quad \ell = 1,2,3 \text{ and } \delta \leq \dfrac{4 V}{3\pi (m - 3)^3}.
  \end{equation}
\end{lemma}
\begin{proof}
  By definition, $q_k^{(\ell)}$ is a $C^2$ piecewise cubic polynomial, and thus absolutely continuous differentiable until the second derivative.
  The third derivative is piecewise constant thus of bounded variation $V_k^{\ell}$.
  Note that $V_k^{\ell}$ is uniquely determined by the spline knots (here the uniform grid $\{ t_{i_\ell}^{(\ell)} \}_{i_\ell=1}^{n_\ell}$) and
  the spline values ${\bf a}_k^{(\ell)}(i_\ell)$.
  The results then follow from the standard error bound for Chebyshev interpolation of differentiable functions, see e.g., \cite[Thm. 7.2]{trefethen2019}.
\end{proof}

Now, we consider the error introduced by the
RHOSVD truncation \cite{khoromskij2009multigrid} in Algorithm \ref{alg:cheb_tuck_alg_cp}, namely, the second term in the right-hand side of \cref{eq:func_cp_error_triangle}.
\begin{lemma}\label{lem:ftilde-fhat}
  Let $\sigma_k^{(\ell)}$, $k=1,\ldots, \min\{m,R\}$ be the descending singular values of the side matrices $C^{(\ell)} = [{\bf c}_1^{(\ell)}, \cdots, {\bf c}_R^{(\ell)}] \in \R^{m \times R}$ of the CP format \cref{eq:cheb_coeff_cp}, and $r_\ell$ be the truncation rank of the RHOSVD approximation, then we have
  \begin{equation}
    \| \hat{f}_{\bf m} - \tilde{f}_{\bf m} \|_{\infty} \leq m^{d/2} \| \boldsymbol{\xi} \| \sum_{\ell=1}^{d} \Big( \sum_{k=r_\ell+1}^{\min\{m,R\}} (\sigma_k^{(\ell)})^2 \Big)^{1/2}, \text{ where } \| \boldsymbol{\xi} \|^2 = \sum_{k=1}^{R} \xi_k^2.
  \end{equation}
\end{lemma}
\begin{proof}
  By similar arguments as in the proof of \cref{prop:cheb_tuck_error}, we have
  $\| \tilde{f}_{\bf m} - \hat{f}_{\bf m} \|_{\infty} \leq m^{d/2} \| {\bf C} - \hat{\bf{C}} \|_{\mathrm{F}}$.
  The result then follows from the standard error analysis of the RHOSVD truncation \cite[Thm. 2.5]{khoromskij2009multigrid}. See also \cite{khoromskaia2022ubiquitous} for a more recent presentation.
\end{proof}
Combining \cref{cor:ftilde-F}, \cref{lem:ftilde-fhat} and the triangle inequality \cref{eq:func_cp_error_triangle}, we arrive at the following theorem immediately, which provides a bound for the error of the ChebTuck approximation in the algebraic case of CP tensor input.
\begin{theorem}\label{thm:cheb_tuck_error_cp}
  With the notations above, and for $\delta d < 1$, the error of the ChebTuck approximation $\hat{f}_{\bf m}$ of a CP tensor ${\bf F}$ produced by \cref{alg:cheb_tuck_alg_cp} is bounded by
  \begin{equation*}
    \err(\hat{f}_{\bf m}, {\bf F}) \leq \| \boldsymbol{\xi} \|_1 \e \delta d + m^{d/2} \| \boldsymbol{\xi} \| \sum_{\ell=1}^{d} \Big( \sum_{k=r_\ell+1}^{\min\{m,R\}} (\sigma_k^{(\ell)})^2 \Big)^{1/2}.
  \end{equation*}
\end{theorem}
\begin{remark}\label{rem:cheb_tuck_error_cp}
  From the above theorem, one can design an adaptive algorithm that constructs the ChebTuck approximation of a CP tensor with a given error $\varepsilon > 0$ by choosing the smallest possible $r_\ell$ in each mode $\ell = 1,2,3$ such that the error bound is less than $\varepsilon$. This is because all the terms in the bound are computable. Specifically, $\delta$ will be anyway computed when we construct the univariate Chebyshev interpolants, and the singular values $\sigma_k^{(\ell)}$ are available after the SVD decomposition of the side matrices in the CP format. Moreover, by \cref{lem:cheb_spline_error}, a sufficient condition for $\delta d < 1$ is
  $m > 3 + (4Vd / 3\pi)^{1/3}$,
  which gives an explicit lower bound on the Chebyshev degree in terms of the spline-regularity parameter $V$.
  The bounded variation $V$ defined in \cref{lem:cheb_spline_error} measures the cumulative jump of the piecewise-constant third derivative across the spline intervals, so it quantifies how oscillatory the cubic spline is at the scale resolved by the grid.
\end{remark}

\subsubsection{The case of multi-particle potentials}\label{sec:cheb_tuck_error_cp_long_range}
In this section, we consider the ChebTuck approximation error for multi-particle potentials of the form
\begin{equation}\label{eq:total_potential}
  P({\bf x}) = \sum_{\nu=1}^N z_\nu p(\|{\bf x}-{\bf x}_\nu\|), \quad z_\nu \in \R \text{ and } {\bf x}_\nu, {\bf x} \in [-1,1]^d, \ d = 3,
\end{equation}
where ${\bf x}_\nu$ is the position of the $\nu$-th particle, $z_\nu$ is its charge and $p(\|{\bf x}\|) = 1/\|{\bf x}\|$ is the Newton (Coulomb) kernel.
This is also the main application area in our numerical experiments. We show that with the use of the range-separated tensor format~\cite{benner2018range}, the error of the ChebTuck approximation of the long-range part decays exponentially fast with respect to the truncation rank $r_\ell$, i.e., we have the following \cref{cor:cheb_tuck_error_cp_long_range} for the second term of the error bound in \cref{thm:cheb_tuck_error_cp}.

\begin{corollary}\label{cor:cheb_tuck_error_cp_long_range}
  Suppose ${\bf F} \in \R^{n\times n \times n}$ is the tensor representation of the long-range part of an $N$-particle potential satisfying the conditions of \cite[Thm. 3.1]{benner2018range}.
  Then the singular values of the side matrices $C^{(\ell)}$ decay exponentially fast, resulting in the following error bound
  \begin{equation} \label{eq:bound_sv}
    \sum_{\ell=1}^d\Big( \sum_{k=r+1}^{\min\{m,R\}} (\sigma_k^{(\ell)})^2 \Big)^{1/2} \leq C_{m,n} N d \exp\left(-\exp\left(r^{\frac{2}{3}}\right)\right),
  \end{equation}
  where, given the target function $f$, the constant $C=C_{m,n}$ depends only on the Chebyshev degree $m$ and the size $n$ of the uniform grid. In particular, the $\varepsilon$-Tucker rank of the canonical CCT tensor ${\bf C}$ is estimated by $\mathcal{O}(\log^{3/2}|\log(\varepsilon / C_{m,n}N)|)$.
\end{corollary}
\begin{proof}
  Let $A^{(\ell)} = [{\bf a}_1^{(\ell)}, \cdots, {\bf a}_R^{(\ell)}] \in \R^{n \times R}$ be the side matrix of the CP tensor ${\bf F}$ \cref{eq:func_val_cp}.
  Then \cite[Thm. 3.1]{benner2018range} implies that the $\varepsilon$-rank of the side matrix $A^{(\ell)}$ has the bound $r_\varepsilon(A^{\ell}) \leq C \log^{3/2}(|\log(\varepsilon / N)|)$ thus we have
  \begin{equation}
  \Big( \sum_{k=r_\ell+1}^{\min\{n,R\}} (\sigma_k^{(\ell)}(A^{(\ell)}))^2 \Big)^{1/2} \leq C N \exp\left(-\exp\left(r^{\frac{2}{3}}\right)\right).
  \end{equation}
  Let $P^{(\ell)}: \R^n \to \R^m$, ${\bf q}_k^{(\ell)} = P^{(\ell)} {\bf a}_k^{(\ell)}$ be the linear operator that maps the canonical vector ${\bf a}_k^{(\ell)}$ to the evaluations of the spline $q_k^{(\ell)}$ at the Chebyshev points, i.e., the procedure in line~\ref{line:spline_eval} of \cref{alg:cheb_tuck_alg_cp}.
  Therefore, we have ${\bf c}_k^{(\ell)} = W^{(\ell)} P^{(\ell)} {\bf a}_k^{(\ell)}$ by line~\ref{line:cheb_coeff_cp} of \cref{alg:cheb_tuck_alg_cp}, hence ${C}^{(\ell)} = W^{(\ell)} P^{(\ell)} A^{(\ell)}$. Therefore we have $\sigma_k^{(\ell)} \leq \|W^{(\ell)} P^{(\ell)}\| \sigma_k^{(\ell)}(A^{(\ell)})$ for $k = 1, \ldots, R$ and $\ell = 1,2,3$ by the min-max characterization of singular values. It is clear that $\|W^{(\ell)} P^{(\ell)}\|$ only depends on $m$ and $n$ since the spline knots and the Chebyshev points are only defined by $m$ and $n$ in the fixed interval $[-1,1]$.
  Absorbing the constant $C$ and the factor $\|W^{(\ell)} P^{(\ell)}\|$ into $C_{m,n}$ completes the proof.
\end{proof}
\begin{remark}\label{rem:cheb_tuck_error_cp_long_range}
  Regarding the constant $C_{m,n}$, numerical estimates show that it is in fact decaying polynomially with respect to both $m$ and $n$, though we do not have a proof for this claim. Combining \cref{cor:cheb_tuck_error_cp_long_range} with \cref{thm:cheb_tuck_error_cp} and \cref{lem:cheb_spline_error}, we can conclude that for the long-range part of the $N$-particle potential, the ChebTuck approximation error has the following estimate
  \begin{equation}
    \err(\hat{f}_{\bf m},{\bf F}) \lesssim d m^{-3} + m^{d/2} N d \exp\left(-\exp\left(r^{\frac{2}{3}}\right)\right).
  \end{equation}
  This estimate holds when $\delta d < 1$, which is guaranteed by choosing the Chebyshev degree $m \geq m_0 = \mathcal{O}((Vd)^{1/3})$ (see also \cref{rem:cheb_tuck_error_cp}).
  For the structured long-range inputs considered here, the Gaussian-integral form of the canonical vectors (see e.g., \cite{benner2018range,bertoglio2012low,hackbusch2006low}) and the numerical evidence suggest that the corresponding variation parameter $V$ behaves essentially independent of $n$. A rigorous analysis of this application-specific behavior, together with explicit constants, is beyond the scope of the present paper and is left for future work.
  For fixed dimension $d$, in particular for practical applications to multi-particle potentials where $d=3$, the doubly exponential decay rate in the truncation rank $r$ essentially compensates the growth $m^{d/2}$ in the second term, resulting in a very favorable error bound.
  The results are consistent with our numerical experiments in \cref{sec:RS_tensor}.
\end{remark}

\subsection{Complexity bounds for \texorpdfstring{\cref{alg:cheb_tuck_func,alg:cheb_tuck_alg,alg:cheb_tuck_alg_cp}}{Algorithms 3.1 to 3.3}}\label{ssec:complexity}

In this section, we briefly discuss the time complexity of the proposed algorithms for the ChebTuck approximation of a function in both functional and algebraic cases.

\begin{proposition}\label{prop:compl}
  For the functional case \cref{alg:cheb_tuck_func}, let $C_{f}$ be the cost of one functional call.
  For the algebraic cases \cref{alg:cheb_tuck_alg,alg:cheb_tuck_alg_cp}, let $n = \max_\ell n_\ell$ be the maximum grid size.
  For the algebraic cases with CP tensor input \cref{alg:cheb_tuck_alg_cp,alg:cheb_tuck_alg_cp}, let $R$ be the rank of the CP tensor. Let $m$ be Chebyshev degree in each mode and $\varepsilon > 0$ be the Tucker-ALS or RHOSVD tolerance. Then the complexity of the algorithms is estimated by \\
  \cref{alg:cheb_tuck_func}: $C_{f} m^3 + C_1 m^3 \log m + C_2 m^4 + C_3 m^4 |\log \varepsilon|$.\\
  \cref{alg:cheb_tuck_alg}: $C_4 m^3 + C_5n^3 + C_1 m^3 \log m + C_2 m^4 + C_3 m^4 |\log \varepsilon|$.\\
  \cref{alg:cheb_tuck_alg_cp}: $C_5 nR + C_6 mR + C_7 m R\log m + C_{8} mR \min\{m,R\} + C_{9} R |\log \varepsilon|^3$.\\
\end{proposition}
\begin{proof}
  For the functional case \cref{alg:cheb_tuck_func}, the number of functional calls is $m^3$ and therefore the cost of evaluating ${\bf T}$ is $C_{f} m^3$. The cost of computing the CCT ${\bf C}$ is $C_1 m^3 \log m$ by the FFT-based Chebyshev coefficient computation. The cost of Tucker-ALS is $C_2 m^4 + C_3 m^4 r$ by the standard results on its complexity, see e.g. \cite[Chapter 3.4.3]{Khor-book_2018}, where the Tucker rank $r = \mathcal{O}(|\log \varepsilon|)$ is assumed. This assumption is true for a wide range of functions including Newton kernel in our numerical experiments \cite{khoromskij2007low}. This completes the proof for the functional case.

  For the algebraic cases, we first claim that the cost of evaluating a cubic spline interpolation at $m$ points given data on uniform $n$-point grid is $\mathcal{O}(m+n)$. Indeed, the computation of the spline coefficients requires solving a tridiagonal system of size $4n-4$, which costs $\mathcal{O}(n)$, and an evaluation of the spline at $m$ points costs $\mathcal{O}(m)$ in case the $n$-point grid is uniform. Similarly, the cost of trivariate cubic spline interpolation which is required in \cref{alg:cheb_tuck_alg} is $\mathcal{O}(m^3+n^3)$. This then completes the proof for the complexity of \cref{alg:cheb_tuck_alg}.

  Finally, for the complexity of \cref{alg:cheb_tuck_alg_cp}, the first two terms correspond to the cost of line~\ref{line:spline_eval}. The computation of the canonical vectors ${\bf c}_k^{(\ell)}$ in line~\ref{line:cheb_coeff_cp} requires $\mathcal{O}(m \log m)$ operations by using FFT, explaining the third term in the estimate. The complexity of the RHOSVD truncation is the sum of the costs of the truncated SVD of the side matrices, amounting to $\mathcal{O}(m R \min\{m,R\})$, and the cost of reconstruction of the Tucker core tensor $\boldsymbol{\beta}$, which is $\mathcal{O}(R r^3)$, see \cite{khoromskaia2022ubiquitous} for more details.
\end{proof}
We notice that there is the traditional trade-off between the accuracy and computational complexity of the numerical algorithms. With regard to this concern, we comment that given $\varepsilon >0$, the polynomial degree $m=m(\varepsilon)$
depends on the regularity of input function $f(x)$ and may increase
dramatically for non-regular functions having, for example,
multiple local singularities or cusps, see for example \cref{tab:single_newton_error}.
This is the case for interaction potentials of many-particle systems typically arising in quantum computational chemistry and bio-molecular modeling.
For such challenging problems, the successful application of our method is based on the use of range-separated (RS) tensor formats, which will be discussed in what follows.

\section{Applications to multi-particle modeling}
\label{sec:app_multi_part_mod}

In this section, we apply our hybrid tensor format to the challenging problem of approximating multi-particle interaction potentials. We consider the total potential generated by an $N$-particle system in $\mathbb{R}^3$, given by a weighted sum \cref{eq:total_potential} of a radially symmetric kernel $p(\|{\bf x}\|)$ centered at the position of each particle ${\bf x}_\nu$.
These kernels are typically slowly decaying and singular at the origin.
Examples include the Newton (Coulomb) $1/\|{\bf x}\|$, Slater $\e^{-\lambda \|{\bf x}\|}$, and Yukawa $\e^{-\lambda \|{\bf x}\|}/\|{\bf x}\|$ potentials,
as well as Mat\'ern covariance functions arising in various applications.
Such practically interesting kernels have local singularities or cusps at the particle centers, and at the same time possess a non-local behavior that is responsible for the global interaction between all the particles in the system.
Due to this short-long range behavior of the total collective potential $P({\bf x})$ it does not allow for a low-rank tensor approximation in traditional tensor formats with satisfactory accuracy. This difficulty can be resolved by applying the RS tensor format~\cite{benner2018range}.

In what follows, we will focus on the ubiquitous Newton kernel $p(\|{\bf x}\|) = 1/\|{\bf x}\|$ for $d=3$, arising in electronic structure calculations and in bio-molecular modeling. We first recall shortly the results on a sinc-quadrature based constructive CP tensor representation of a single Newton kernel discretized on a 3D Cartesian grid \cite{hackbusch2006low,bertoglio2012low}, thereby providing the CP tensor input for our mesh-free ChebTuck approximation \cref{alg:cheb_tuck_alg_cp}. We then introduce the idea of range-separation for efficient numerical treatment of the single Newton kernel \cite{benner2018range}. Finally, we discuss how the total potential \cref{eq:total_potential} can be constructed from the CP tensor representation of the single Newton kernel, for both bio-molecular modeling and lattice-structured compounds with vacancies, for which cases our ChebTuck approximation can be gainfully applied.

\subsection{ChebTuck approximation for rank-structured tensor representation of the Newton kernel}
\subsubsection{A low-rank CP tensor representation of the Newton kernel}

Here we recall the constructive low-rank representation of the Newton kernel based on sinc-quadrature, as developed in~\cite{hackbusch2006low,bertoglio2012low}, and refer to these works for the detailed error analysis and numerically optimized procedure for rank minimization.

Let ${\bf P} \in \R^{n_1\times n_2 \times n_3}$ be the projection-collocation tensor of the Newton kernel $p({\bf x}) = 1/\|{\bf x}\|$ discretized on a $n_1\times n_2\times n_3$ 3D uniform Cartesian grid in $[-1,1]^3$, where the entries are defined by (see \cite{hackbusch2006low} for more details)
\begin{equation}\label{eq:def_newton_tensor}
  {\bf P}_{{\bf i}}
  = \int_{[-1,1]^3} \dfrac{\psi_{{\bf i}}({\bf x})}{\|{\bf x}\|} \dd {\bf x} = \dfrac{1}{h_1 h_2 h_3}\int_{t_{i_1}^{(1)}}^{t_{i_1+1}^{(1)}} \int_{t_{i_2}^{(2)}}^{t_{i_2+1}^{(2)}} \int_{t_{i_3}^{(3)}}^{t_{i_3+1}^{(3)}} \dfrac{1}{\|{\bf x}\|} \dd x_1 \dd x_2 \dd x_3,
\end{equation}
with the multi-index  ${\bf i} = (i_1, i_2, i_3)$, where $t_{i_\ell}^{(\ell)} = -1 + (i_\ell-1) h_\ell $ are the grid points in each dimension and $h_\ell = 2/n_\ell$ is the grid size.
The basis function $\psi_{{\bf i}}$ is a product of piecewise constant functions in each dimension $\psi_{{\bf i}}({\bf x}) = \psi_{i_1}(x_1) \psi_{i_2}(x_2) \psi_{i_3}(x_3)$, where $\psi_{i_{\ell}}(x_{\ell}) = \chi_{[t_{i_{\ell}}^{(\ell)},t_{i_{\ell}+1}^{(\ell)})}(x_{\ell})/h_\ell$.
To ease notation, we assume that $n_1 = n_2 = n_3 = n$ in the following, and therefore use $h$ and $t_{i_\ell}$ to denote the grid size and grid points in all dimensions respectively. A numerically optimized procedure for rank minimization of the sinc-based CP approximation of ${\bf P}$ is then applied, leading to the final rank-$R$ CP format ${\bf P}_R$ of the collocation tensor ${\bf P}$
\begin{equation}\label{eq:newton_cp}
  {\bf P} \approx {\bf P}_R= {\sum}_{k = 1}^R
  {\bf p}_k^{(1)} \otimes {\bf p}_k^{(2)}\otimes {\bf p}_k^{(3)}
\end{equation}
where ${\bf p}_k^{(\ell)}$ are canonical vectors. We remark here that the canonical vectors in each dimension are the same due to the radial symmetry of the Newton kernel, i.e., ${\bf p}_k^{(1)} = {\bf p}_k^{(2)} = {\bf p}_k^{(3)} = {\bf p}_k$.  Note that here, the rank $R$ is adapted to the chosen approximation error and to the corresponding grid size $n$, and scales only logarithmically in both parameters. Similarly, the CP tensor decomposition can be constructed for a wide class of analytic kernels as described in~\cite{hackbusch2006low,bertoglio2012low}.
In \cref{fig:cp_vectors} left, we plot the canonical vectors ${\bf p}_k^{(1)}$ for $k = 1, \ldots, R$ of the CP approximation ${\bf P}_R$ of ${\bf P}$ discretized on a $n\times n\times n$ Cartesian grid for $n=256$, presented for example in \cite{Khor-book_2018,Khor_bookQC_2018}. In this case, $R = 26$ gives a relative error in the order of $10^{-7}$.
We can see that there are canonical vectors localized around the singularity at the origin as well as slowly decaying ones, which clearly corresponds to the long-range and short-range behavior of the Newton kernel.
We remark here that the number of vectors $R$ depends logarithmically on the size of the $ n\times n\times n$ 3D Cartesian grid and the parameter $\varepsilon>0$ controlling the accuracy of the sinc-approximation. For increasing grid size $n$, the number of vectors with localized support around the singularity at the origin becomes larger.
\captionsetup{belowskip=0pt}
\begin{figure}
  \centering
  \includegraphics[width=0.49\textwidth]{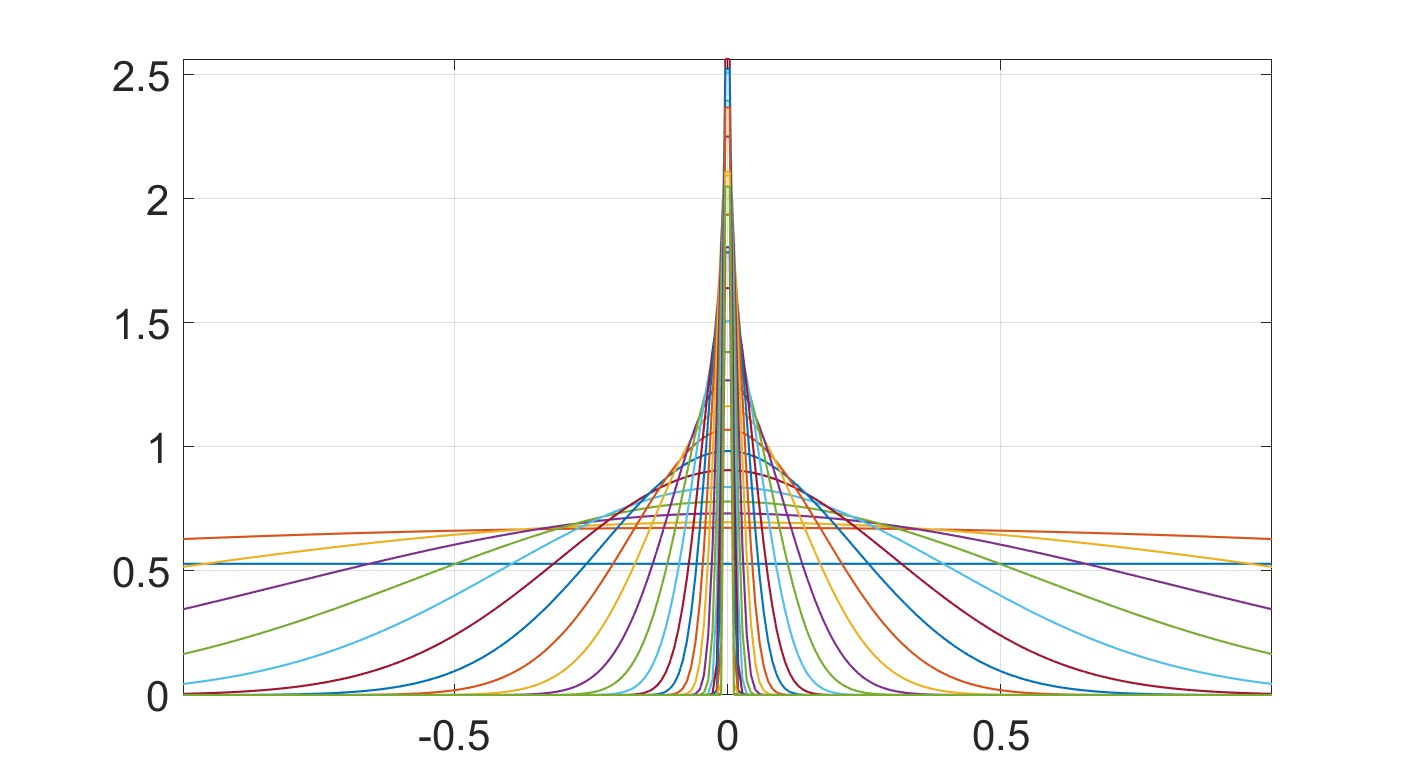}
  \includegraphics[width=0.49\textwidth]{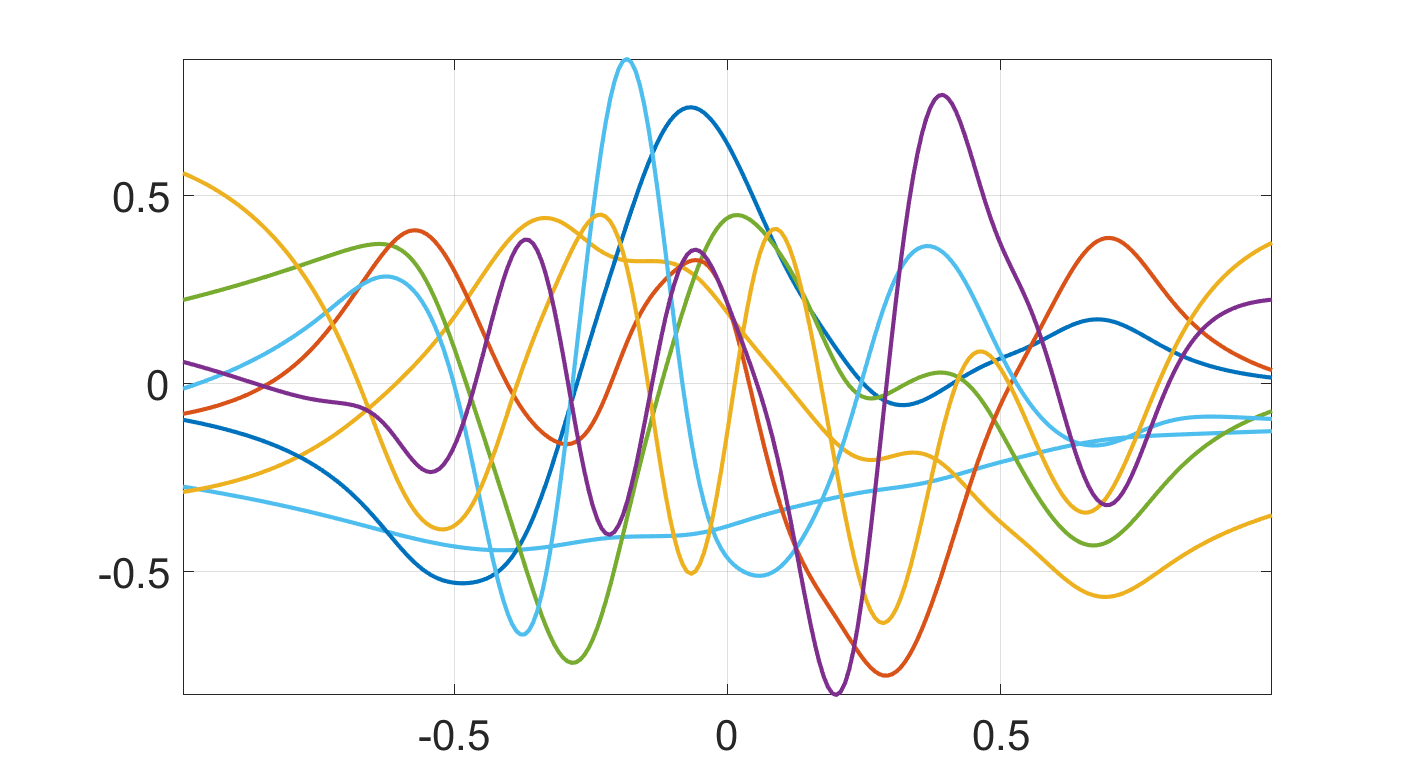}
  \caption{{\bf Left}: the canonical vectors ${\bf p}_k^{(1)}$ for $k = 1, \ldots, R$, $R=26$, along the $x$-axis for the rank-$R$ CP tensor approximation of collocation tensor of size ${\bf P}_R \in \R^{n \times n \times n}$ with $n = 256$.
  {\bf Right}: canonical vectors ${\bf a}_k^{(1)}$ of the long-range part in the total potential ${\bf A}_{R_l} \in \R^{n\times n \times n}$ of a biomolecular fragment (see \cref{sec:RS_tensor} for details) with $N=500$ particles and grid size $n=256$ for $k = 1, \ldots, 8$.
  }
  \label{fig:cp_vectors}
\end{figure}

This low-rank CP tensor representation of the Newton kernel has been successfully used by two of the coauthors in high accuracy electronic structure calculations
\cite{khoromskij2009multigrid,Khor_bookQC_2018}. In particular, it was applied for
tensor based computation of the Hartree potential (3D convolution
of the electron density with the Newton kernel), the nuclear
potential operators, and the two-electron integral tensor,
see \cite{Khor_bookQC_2018} and the references therein.
All computations have been performed using huge 3D
$n\times n\times n$ Cartesian grids with  $n$ of the order of
$10^4\sim 10^5$, providing accuracy close to the results from the analytically
based quantum chemical packages.

\subsubsection{Why Chebyshev interpolation fails to approximate \texorpdfstring{${\bf P}_R$}{PR} without range-separation}
\label{sec:cheb_tuck_fail}

In this subsection, we consider the ChebTuck approximation \cref{alg:cheb_tuck_alg_cp} applied to the CP tensor ${\bf P}_R$ of the singular Newton kernel and discuss why it fails to give satisfactory results, which necessitates the use of the range-separated (RS) tensor format \cite{benner2018range}.

First, we compute the ChebTuck approximation relative error of the CP tensor ${\bf P}_R$ at the middle slice for different grid size $n$ and Chebyshev degree $m$, i.e.,
\begin{equation}\label{eq:err_single_newton}
  \mathrm{err} := \max_{i_1,i_2} \left| {\bf P}_R(i_1, i_2, n/2) - \hat{f}_{\bf m}(t_{i_1}, t_{i_2}, t_{n/2}) \right| / \max_{i_1,i_2} |{\bf P}_R(i_1, i_2, n/2)|,
\end{equation}
as shown in \cref{tab:single_newton_error}. We can see that in this case, only when $m\gg n$ the ChebTuck approximation provides satisfactory results, which is not practical and does not reduce the storage cost.
\begin{table}
  \centering
  \begin{tabular}{ccccccccccc}
    \toprule
    \diagbox{$n$}{$m$} & 129  & 257  & 513    & 1025   & 2049   & 4097   & 8193   & 16385  \\
    \midrule
    256                & 0.41 & 0.07 & 7.7e-3 & 6.9e-4 & 2.0e-4 & 8.5e-6 & 1.6e-6 & 3.5e-7 \\
    512                & 0.71 & 0.41 & 0.07   & 7.7e-3 & 6.9e-4 & 2.0e-4 & 8.5e-6 & 1.6e-6 \\
    1024               & 0.92 & 0.71 & 0.41   & 0.07   & 7.7e-3 & 6.9e-4 & 2.0e-4 & 8.5e-6 \\
    2048               & 1.06 & 0.92 & 0.71   & 0.41   & 0.07   & 7.7e-3 & 6.9e-4 & 2.0e-4 \\
    4096               & 1.16 & 1.06 & 0.92   & 0.71   & 0.41   & 0.07   & 7.7e-3 & 6.9e-4 \\
    \bottomrule
  \end{tabular}
  \caption{Relative $\ell_\infty$ error in the middle slice of the ChebTuck approximation of the single Newton kernel as in \cref{eq:err_single_newton} for different grid size $n$ and Chebyshev degree $m$.}
  \label{tab:single_newton_error}
\end{table}
To have a closer look at the behavior, we take the collocation tensor of a single Newton kernel ${\bf P}_R \in \R^{n\times n\times n}$ with grid size $n = 8192$ as an example. In this case, we get a rank $R = 39$ CP approximation with relative error $10^{-7}$. We compute univariate Chebyshev interpolants of all canonical vectors ${\bf p}_k$ as described in \cref{ssec:alg_cp_error} of degree $m$, denoted by $g_k(x)$. For each canonical vector ${\bf p}_k$, we measure the Chebyshev approximation error at the uniform grid points, i.e.,
\begin{equation}\label{eq:err_single_newton_1d}
  \mathrm{err} := \max_{i} \left| {\bf p}_k(i) - g_k(t_i) \right|/ \max_{i} |{\bf p}_k(i)|,
\end{equation}
for different Chebyshev degree $m$. The results are shown in \cref{fig:error_vs_cheb_deg}. We can see that canonical vectors representing the long-range behavior of the Newton kernel, i.e., ${\bf p}_1, \ldots, {\bf p}_{15}$ can be well approximated by low-degree Chebyshev interpolants (left), while for short-range canonical vectors, the convergence is much slower (middle). This is as expected since the short-range canonical vectors have singularities at the origin and require high-degree polynomials to capture the behavior. In \cref{fig:error_vs_cheb_deg} right, we plot the Chebyshev approximation error of different canonical vectors for a fixed Chebyshev degree. We can see given a fixed degree $m=129$, the error of the short-range functions is much larger than for the long-range functions with a jump around $k=15$. The same behavior can also be seen from \cref{fig:ls_coefficients}, in which the absolute values of the Chebyshev coefficients of the interpolant of a long-range vector ${\bf p}_{10}$ and a short-range vector ${\bf p}_{30}$ are shown. We can see that for the long-range vector, the Chebyshev coefficients decay exponentially while for the short-range vector, the coefficients do not decay very much.

This motivates us to perform range-separation of the CP tensor ${\bf P}_R$ and use the range-separated tensor format \cite{benner2018range} first before applying the ChebTuck approximation, which will be discussed in the next subsection.
\begin{figure}
  \centering
  \includegraphics[width=0.32\textwidth]{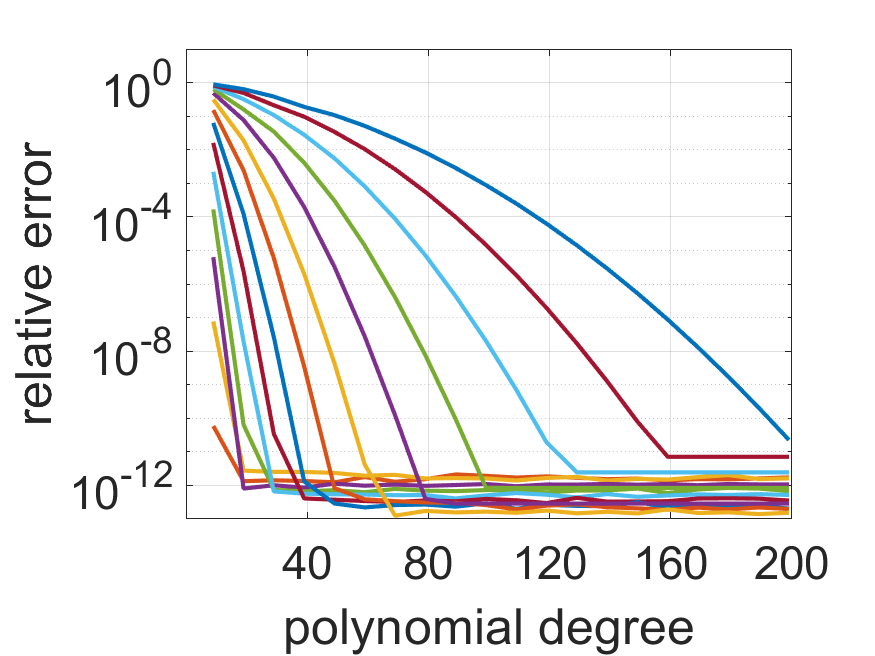}
  \includegraphics[width=0.32\textwidth]{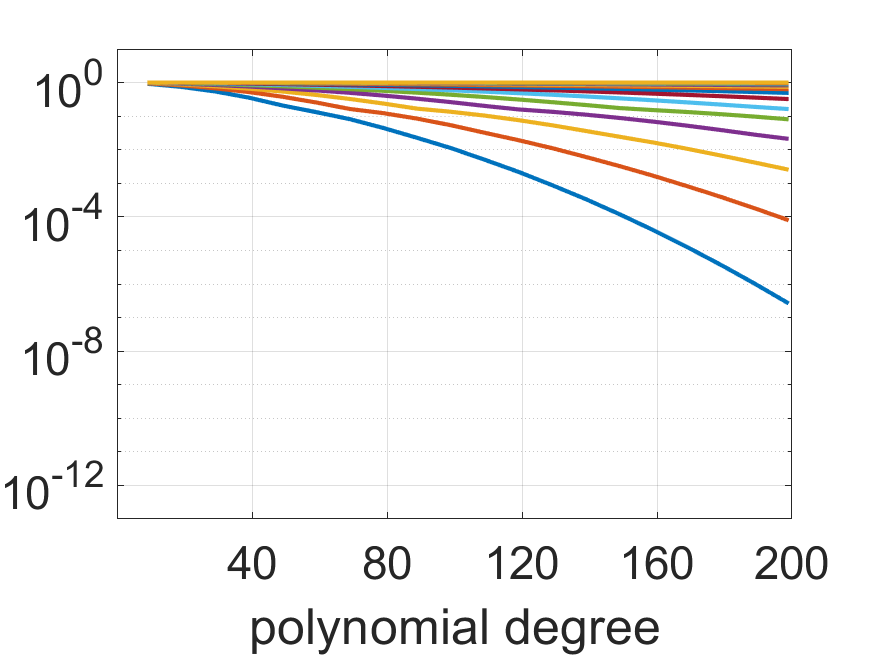}
  \includegraphics[width=0.32\textwidth]{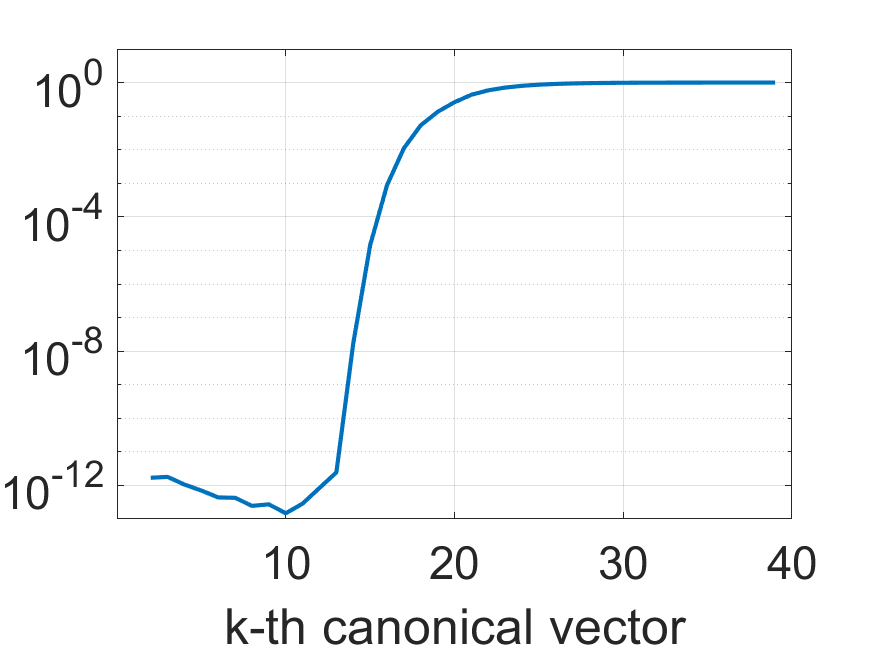}
  \caption{Relative $\ell_\infty$ error \cref{eq:err_single_newton_1d} of the Chebyshev interpolation of the canonical vectors vs. the degree of the Chebyshev interpolant. Each line represents the error of the Chebyshev interpolant of a canonical vector ${\bf p}_k$ for $k = 1, \ldots, 15$ ({\bf left}) and $k = 16, \ldots, 39$ ({\bf middle}). {\bf Right}: Relative $\ell_\infty$ error of the degree $m=129$ Chebyshev interpolation of the canonical vectors vs. the index $k$ of the univariate function.}
  \label{fig:error_vs_cheb_deg}
\end{figure}
\begin{figure}
  \centering
  \includegraphics[width=0.6\textwidth]{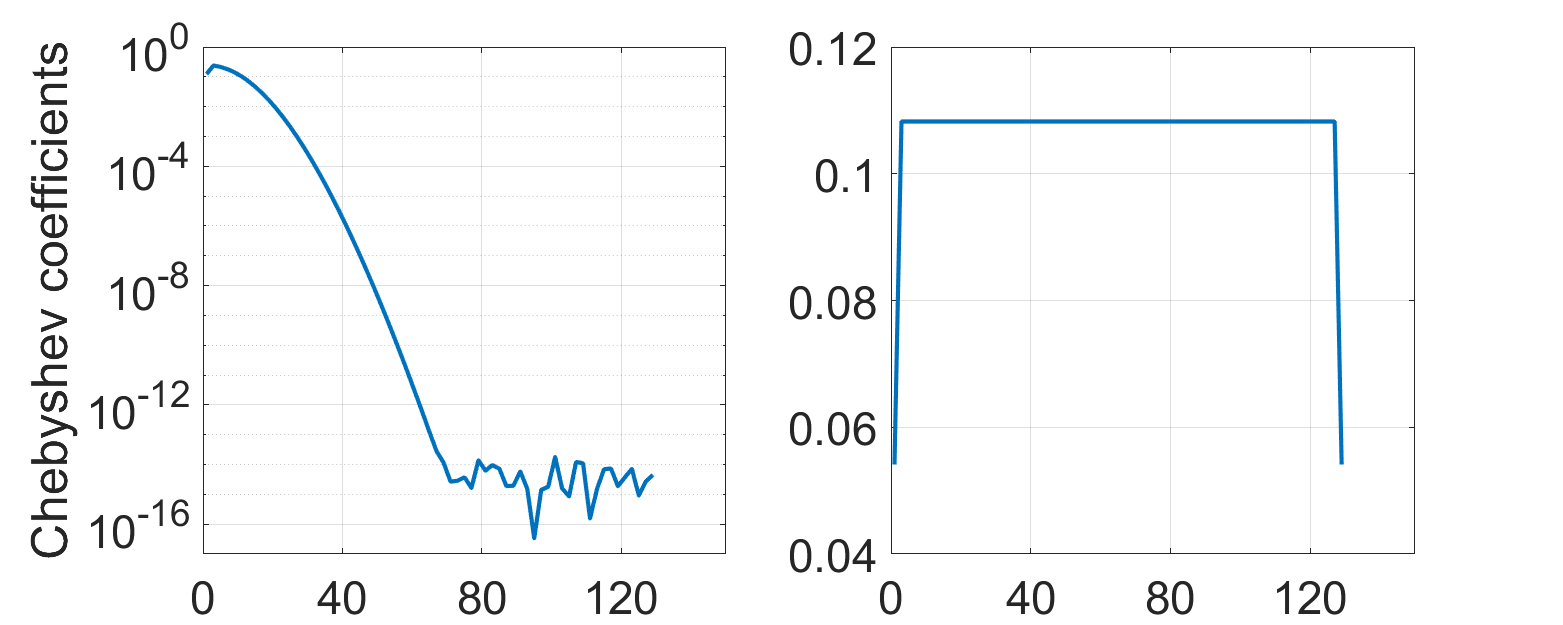}
  \caption{Absolute values of the Chebyshev coefficients of the degree-129 Chebyshev interpolant of ${\bf p}_{10}$ ({\bf left}) and ${\bf p}_{30}$ ({\bf right}). Only coefficients of even degree are shown, i.e., ${\bf c}_k(0), {\bf c}_k(2), \ldots, {\bf c}_k(128)$ since $p(\|{\bf x}\|)$ is even and ${\bf c}_k(1), \ldots, {\bf c}_k(129)= 0$}
  \label{fig:ls_coefficients}
\end{figure}

\subsubsection{Recalling the RS tensor format for \texorpdfstring{${\bf P}_R$}{PR}}\label{sec:rs_single_newton}
For a given appropriately chosen range-separation parameter $R_s \in [1, R] \cap \Z$ (see \cite{benner2018range} for details of the strategies for choosing $R_s$), we further use the
range-separated representation of the CP tensor ${\bf P}_R$  constructed in \cite{benner2018range} in the form ${\bf P}_R = {\bf P}_{R_s} + {\bf P}_{R_l}$, where
\begin{equation}\label{eqn:Split_Tens}
  {\bf P}_{R_s} = {\sum}_{k=1}^{R_s} {\bf p}_k \otimes {\bf p}_k \otimes {\bf p}_k, \  {\bf P}_{R_l} = {\sum}_{k=R_s+1}^{R} {\bf p}_k \otimes {\bf p}_k \otimes {\bf p}_k.
\end{equation}
Here, \cref{eqn:Split_Tens} represents the sums over
the sets of indexes for the long- and short-range canonical vectors in the range-separation splitting of the Newton kernel. Then the short-range part ${\bf P}_{R_s}$ is a highly localized tensor which has very small support and can thus be stored sparsely (see \cite{benner2018range} for details), while the ChebTuck approximation can be gainfully applied to the long-range part ${\bf P}_{R_l}$.
In the single-kernel experiments reported below, we use the splitting parameter $R_l = 16$ and hence $R_s = 23$ determined by the range separation principles described in \cite[Eqs 2.12 - 2.13]{benner2018range}. More detailed discussions on the choice of the range-separation parameter $R_s$ can be found in \cite{benner2018range}.

We apply \cref{alg:cheb_tuck_alg_cp} to the long-range part ${\bf P}_{R_l}$ and measure the errors of the approximation at the middle slice as in \cref{eq:err_single_newton} for different grid size $n$ and Chebyshev degree $m$, which are shown in \cref{tab:lr_single_newton_error}.
As opposed to \cref{tab:single_newton_error}, we can see that the ChebTuck approximation of the long-range part ${\bf P}_{R_l}$ can give satisfactory results already for reasonable Chebyshev degrees $m$.
We also show the optimal Tucker ranks, i.e., size of the core tensor $\boldsymbol{\beta}$ of the ChebTuck approximation of the long-range part ${\bf P}_{R_l}$ computed by RHOSVD with a relative error of $10^{-7}$ for different grid size $n$ and Chebyshev degree $m$ in the last column of \cref{tab:lr_single_newton_error}.
For fixed Chebyshev degree, the Tucker rank of the long-range part only increases logarithmically with the grid size $n$.
This indicates that the long-range part can be well approximated by a ChebTuck format with a small number of parameters.

\begin{remark}\label{rem:long_range-ChebTuck_vs_FMM_Newton}
As a side note, we remark that the numerical evidence in \cref{sec:cheb_tuck_fail,sec:rs_single_newton} aligns with the intuition underlying the fast multipole method (FMM) \cite{greengard1987fast}: the far-field (long-range) component of the potential admits an accurate approximation by low-degree polynomials.
\end{remark}
\begin{table}
  \centering
  \begin{tabular}{cccccccccc}
    \toprule
    \diagbox{$n$}{$m$} & 129    & 257     & 513     & 1025    & 2049    & 4097    & Tucker rank \\
    \midrule
    256                & 8.4e-8 & 7.4e-8  & 7.3e-8  & 4.9e-9  & 1.6e-9  & 1.8e-10 & 9 \\
    512                & 5.2e-8 & 5.2e-8  & 4.6e-8  & 4.6e-8  & 3.0e-9  & 9.9e-10 & 11 \\
    1024               & 9.4e-9 & 9.4e-9  & 9.4e-9  & 8.7e-9  & 8.6e-9  & 5.7e-10 & 12 \\
    2048               & 3.6e-9 & 1.7e-9  & 1.7e-9  & 1.7e-9  & 1.6e-9  & 1.6e-9  & 13 \\
    4096               & 1.1e-4 & 7.5e-10 & 7.4e-10 & 7.4e-10 & 7.4e-10 & 7.1e-10 & 15 \\
    \bottomrule
  \end{tabular}
  \caption{Long-range part of the Newton potential. Relative $\ell_\infty$ error as in \cref{eq:err_single_newton} and in the middle slice and the Tucker rank $\max_{\ell=1,2,3}{r_{i_\ell}}$ of the ChebTuck approximation of the long-range part ${\bf P}_{R_l}$ for different grid sizes $n$ and degree $m$. The Tucker rank is the same for different $m$ for a fixed grid size $n$, therefore only one value is shown.}
  \label{tab:lr_single_newton_error}
\end{table}
\subsection{Application of ChebTuck approximation to RS tensor format}\label{sec:RS_tensor}

In this section, we first recall the construction of the range-separated tensor format \cite{benner2018range} for the total potential \cref{eq:total_potential} in an $N$-particle system. Then we consider two practical examples, a bio-molecular system and a lattice-structured compound with vacancies, where we apply the ChebTuck approximation to the long-range part of the total potential.

The total potential of an $N$-particle system $P({\bf x})$ in \cref{eq:total_potential} is nothing but a weighted sum of shifted Newton kernels.
It was proven in \cite{khoromskaia2014grid} that the tensor representation of the collocation tensor of the total potential can also be obtained by shifting the collocation tensor of the Newton kernel, albeit through a more elaborate construction.
In fact, it can be obtained through the application of shifting and windowing operators to the reference potential in \cref{eq:newton_cp}, see~\cite{khoromskaia2014grid},
${\cal W}_{\nu}={\cal W}_{\nu}^{(1)}\otimes {\cal W}_{\nu}^{(2)}\otimes {\cal W}_{\nu}^{(3)}$, $\nu=1,\cdots,N$,
which applies to every particle located at ${\bf x}_\nu$ in an $N$-particle system.
Numerically it is obtained by construction of the single generating (reference) potential $\widetilde{\bf P}_R \in \R^{2n\times 2n \times 2n}$ in a twice larger computational box and shifting it according to the $x_{\mu,1}, x_{\mu,2}, x_{\mu,3} $ coordinates of every particle ${\bf x}_\mu$ with further restricting (windowing) it to the original computational domain.
Then the total multi-particle potential in \cref{eq:total_potential} is computed as a projected weighted sum of shifted and windowed potentials ${\bf P}_0 = \sum_{\nu=1}^{N} {z_\nu}\, {\cal W}_\nu (\widetilde{\bf P}_R)\in \mathbb{R}^{n\times n \times n}$, see also \cite{Khor_bookQC_2018} for more details.

It was proven in \cite{benner2018range} that a tensor representation of this type in case of a multi-particle system of general type has almost irreducible large rank, $R_N= O(NR)$.
To overcome this drawback, in \cite{benner2018range} a new RS tensor format was introduced based on separation of the long- and short-range skeleton vectors in the generating tensor $\widetilde{\bf P}_R = \widetilde{\bf P}_{R_s} + \widetilde{\bf P}_{R_l}$ similar to the single Newton kernel version as in \cref{sec:rs_single_newton} which leads to a separable representation for the total sum of potentials,
\[
  {\bf P}_0 =
  {\sum}_{\nu=1}^{N} {z_\nu} \, {\cal W}_\nu (\widetilde{\mathbf{P}}_{R_s} + \widetilde{\mathbf{P}}_{R_l})
    =: {\bf P}_s + {\bf P}_l.
\]
Thanks to the RS tensor format, the CP-rank of the part of the sum, $\mathbf{P}_{l}$, containing only the long-range components of the canonical vectors of the shifted kernels, scales as $R_l \in O(R \log N)$, i.e., logarithmically in the number of particles in a molecular system \cite{benner2018range}.
Such rank compression of the long-range part $\mathbf{P}_l$ is achieved by applying the RHOSVD method for the canonical to Tucker (C2T) transform, and the subsequent Tucker to canonical (T2C) decomposition \cite{khoromskij2009multigrid,khoromskij2007low}.
This yields the resulting canonical tensor $\mathbf{P}_{R_l}$, ${\mathbf{P}}_{l} \mapsto {\mathbf{P}}_{R_{l}}$ with the reduced rank $ R N \mapsto R_{l} \in O(R\log N)$.

One can aggregate the sum of short-range tensors ${\bf P}_s$ (having local supports which do not overlap each other) into the so called ``cumulated canonical tensors'', $\mathbf{P}_{s}$, parametrized only by the coordinates of particle centers and parameters of the localized reference tensor, $\widetilde{\mathbf{P}}_{R_s}$.
Above considerations resulted in the following definition (see~\cite{benner2018range} for more details).

\begin{definition}  \label{Def:RS-Can_format}
  (RS-canonical tensors \cite{benner2018range}).
  Given a reference tensor ${\bf A}_0$ supported by a small box such that $\mbox{rank}({\bf A}_0)\leq R_0$,
  the separation parameter $\gamma \in \mathbb{N}$, and a set of distinct points
  ${\bf x}_\nu \in \mathbb{R}^{d}$,
  $\nu=1,\ldots,N$,
  the {\bf RS-canonical tensor format} specifies the class of $d$-tensors
  ${\bf A}  \in \mathbb{R}^{n_1\times \cdots \times n_d}$,
  which can be represented as a sum of a rank-${R}_l$ canonical tensor ${\bf A}_{R_l}$ and a cumulated canonical tensor $\widehat{\bf A}_S$,
  \begin{equation}\label{eq:LR_tensor_sum}
    {\bf A}_{R_l} = {\sum}_{k =1}^{R_L} \xi_k {\bf a}_k^{(1)} \otimes \cdots \otimes {\bf a}_k^{(d)}, \ \widehat{\bf A}_S={\sum}_{\nu =1}^{N} c_\nu {\bf A}_\nu
    \in \mathbb{R}^{n_1\times \cdots \times n_d},
  \end{equation}
  where $\widehat{\bf A}_S$ is generated by replication of the reference tensor ${\bf A}_0$ to the
  points ${\bf x}_\nu$.
  Then the RS canonical tensor is represented in the form ${\bf A} =  {\bf A}_{R_l} + \widehat{\bf A}_S$, where $\mbox{diam}(\mbox{supp}{\bf A}_\nu)\leq 2 \gamma$ in the index size.
\end{definition}

In what follows, we consider biomolecular test systems corresponding to truncated fragments of the protein Fasciculin 1 with different numbers of atoms $N$. Fasciculin 1 is an anti-acetylcholinesterase toxin from the green mamba snake venom comprising 1,228 atoms \cite{PBE_RS:21, DuMaBoFo:92}.
In all biomolecular experiments in \cref{sec:bio_molecule,sec:logscaling}, $N$ denotes the number of atoms in the truncated fragment of the original molecule.
We also consider a lattice-structured compound with vacancies in \cref{sec:lattice}, which is a model system for crystalline solids.
We demonstrate similar accuracy control for the proposed mesh-free two-level tensor approximation applied to sums of Newton kernels for biomolecular fragments as well as for lattice-type compounds.

\begin{figure}[h]
  \centering
  \includegraphics[width=0.32\textwidth]{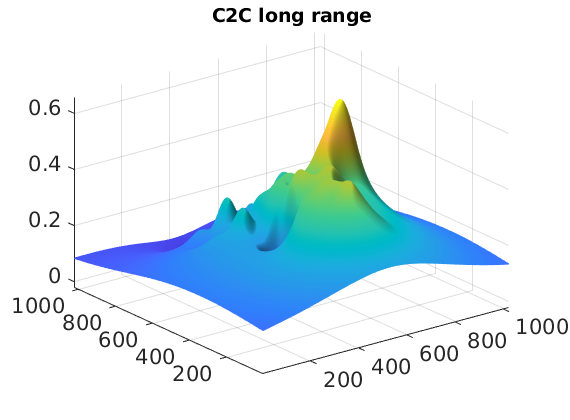}
  \includegraphics[width=0.32\textwidth]{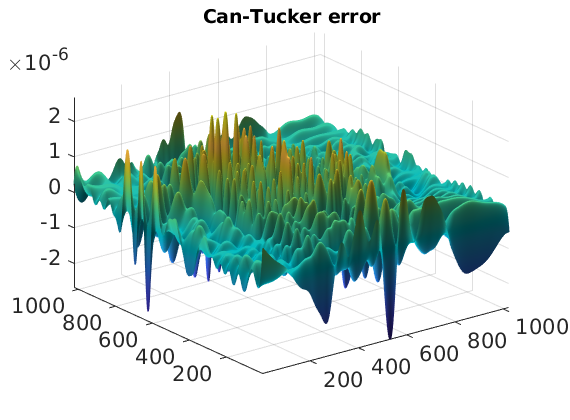}
  \includegraphics[width=0.32\textwidth]{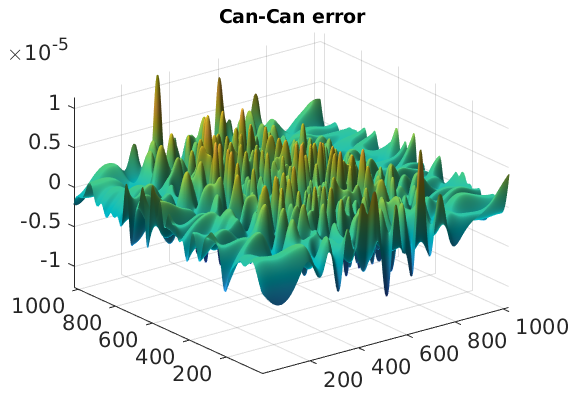}
  \caption{{\bf Left}: Middle slice ${\bf A}_{R_l}(:,:,n/2)$ of the (compressed) long-range part of the RS tensor for the electrostatic potential of
  a truncated Fasciculin 1 fragment with $N=500$ atoms computed on the
  3D $n\times n\times n$ Cartesian grid with $n=1024$ and rank $R_l=506$.
  {\bf Middle}: the error of the canonical-to-Tucker transform,
  with $\varepsilon=10^{-6} $. {\bf Right}: the total error of the
  canonical-to-canonical tensor transform with the $\varepsilon=10^{-5}$.
  }
  \label{fig:RS_tensor}
\end{figure}

\subsubsection{A truncated biomolecular fragment with \texorpdfstring{$N=500$}{N=500} atoms}\label{sec:bio_molecule}
In this section, we consider the ChebTuck approximation to the long-range part of a truncated Fasciculin 1 fragment with $N=500$ atoms. In this case, the underlying function we want to approximate is given by a 3-dimensional canonical tensor ${\bf A}_{R_l} \in \R^{n \times n \times n}$, which can be pre-computed for any given grid size $n$ as described in \cref{sec:RS_tensor}. In this case, \cref{alg:cheb_tuck_func} is not applicable but \cref{alg:cheb_tuck_alg,alg:cheb_tuck_alg_cp} can be used.

In \cref{fig:RS_tensor}, we visualize the middle slice for the long-range part of the RS tensor, see the left figure, representing the total electrostatic potential of the truncated Fasciculin 1 fragment with $N=500$ atoms, i.e., ${\bf A}_{R_l}$ of the $N$-particle potential in CP tensor format, with the reduced canonical rank $R_l=506 \sim R\log(N)$ as compared to the original canonical rank $R_N = 7000 \sim NR$ computed by the RHOSVD-based rank reduction method \cite{khoromskij2009multigrid}. We also show the error of the canonical-to-Tucker transform and of the consequent Tucker-to-canonical transform, ultimately canonical-to-canonical transform, in the middle and right figures in \cref{fig:RS_tensor}, respectively. We observe a good approximation error recovering the accuracy of rank truncation in tensor transforms.
In Figure \ref{fig:cp_vectors} right, we also visualize the several canonical vectors ${\bf a}_k^{(1)}$, $k=1, \ldots, 10 $, in the $x$-direction of the long-range part ${\bf A}_{R_l}$ of the same potential as in Figure \ref{fig:RS_tensor}.

We first consider the runtime and accuracy of \cref{alg:cheb_tuck_alg,alg:cheb_tuck_alg_cp} for different discretization grid sizes $n$. We set the Chebyshev degrees in each dimension to be $m_1 = m_2 = m_3 = m = 129$ and the Tucker-ALS (RHOSVD in case of \cref{alg:cheb_tuck_alg_cp}) truncation tolerance to be $10^{-6}$. We record the runtime and $\ell_\infty$ error at the middle slice for each case in \cref{tab:runtime_approx}. Note that since the trivariate cubic spline in \cref{alg:cheb_tuck_alg} is cubic in grid size $n$ and for each Chebyshev degree $m$ (see \cref{prop:compl}), we only compute the results for $n=256$.
These timings are for the ChebTuck approximation stage once the long-range RS tensor input is available.
We observe that the accuracy of the algorithms is similar while \cref{alg:cheb_tuck_alg_cp} is significantly more efficient as it leverages the CP tensor structure of the long-range part as well as the RHOSVD method for Tucker compression. This is consistent with the theoretical estimates in \cref{prop:compl}.
The ChebTuck format also becomes more accurate with increasing grid size $n$ which provides more information for the underlying function to be approximated.
\begin{table}
  \centering
  \begin{tabular}{ccccc}
    \toprule
             & \multicolumn{2}{c}{\cref{alg:cheb_tuck_alg}} & \multicolumn{2}{c}{\cref{alg:cheb_tuck_alg_cp}}                                   \\
    \midrule
             & $\ell_\infty$ error                          & runtime                                         & $\ell_\infty$ error   & runtime \\
    $n=256$  & $8.59 \cdot 10^{-7}$                         & 9.44                                           & $8.59 \cdot 10^{-7}$  & 0.07    \\
    $n=512$  & \diagbox[height=\line]{}{}                   & \diagbox[height=\line]{}{}                      & $6.89 \cdot 10^{-8}$  & 0.09    \\
    $n=1024$ & \diagbox[height=\line]{}{}                   & \diagbox[height=\line]{}{}                      & $4.21 \cdot 10^{-9}$  & 0.13    \\
    $n=2048$ & \diagbox[height=\line]{}{}                   & \diagbox[height=\line]{}{}                      & $6.37 \cdot 10^{-11}$ & 0.18    \\
    \bottomrule
  \end{tabular}
  \caption{Runtime in seconds and $\ell_\infty$ error for the ChebTuck approximation of the potential of a truncated Fasciculin 1 fragment with $N=500$ atoms in $\mathbb{R}^3$ for different grid sizes $n$.}
  \label{tab:runtime_approx}
\end{table}
Next, we visualize the ChebTuck approximation error as well as the RHOSVD compression error for the potential with $n=256$ as in \cref{fig:chebfun3h_error}. In the left figure, we show the middle slice of the approximated potential $\hat{f}_{\bf m}(:,:,t_{n/2}) \approx {\bf A}_{R_l}(:,:,n/2)$. Note that here we abuse the notation and use $\hat{f}_{\bf m}(:,:,t_{n/2})$ to denote the matrix containing function values of $\hat{f}_{\bf m}$ at the uniform grid points in the middle slice, i.e., $\hat{f}_{\bf m}(:,:,t_{n/2}) = [\hat{f}_{\bf m}(t_{i_1}, t_{i_2}, t_{n/2})]_{i_1 i_2}$. In the middle and right figures, we show the error during RHOSVD compressing and the total ChebTuck approximation error, i.e., $\hat{f}_{\bf m}(:,:,t_{n/2}) - \tilde{f}_{\bf m}(:,:,t_{n/2})$ and $\hat{f}_{\bf m}(:,:,t_{n/2}) - {\bf A}_{R_l}(:,:,n/2)$, respectively. Note the difference between $\hat{f}_{\bf m}$ and $\tilde{f}_{\bf m}$ as described in \cref{sec:ChebTuck}. We observe good approximation accuracy of our ChebTuck format. Another notable observation is that the error during RHOSVD compression is much smaller than the prescribed tolerance $\varepsilon = 10^{-7}$, which indicates that the singular values decay extremely fast after the first few significant ones. Similar figures are also obtained for larger grid sizes $n=512,1024,2048$ as shown in \cref{fig:chebfun3h_error_n}. Our ChebTuck format achieves even higher accuracy for larger grid sizes, consistent with the results in \cref{tab:runtime_approx}.

We also remark here that the choice of the univariate interpolations $q_k^{(\ell)}(x)$ in \cref{eq:cp_2_spline} is crucial for the accuracy of the ChebTuck approximation since the convergence rate of the Chebyshev interpolants $g_k^{(\ell)}(x)$ depends on the smoothness of the univariate interpolations $q_k^{(\ell)}(x)$, as discussed in \cref{lem:cheb_spline_error}. For example, in \cref{fig:chebfun3h_error_nn}, we replace the spline interpolation in \cref{eq:cp_2_spline} by piecewise linear (middle) and nearest neighbor (right) interpolation and observe that the ChebTuck format accuracy deteriorates significantly even for a large grid size $n=2048$.
\begin{figure}
  \centering
  \includegraphics[width=0.32\textwidth]{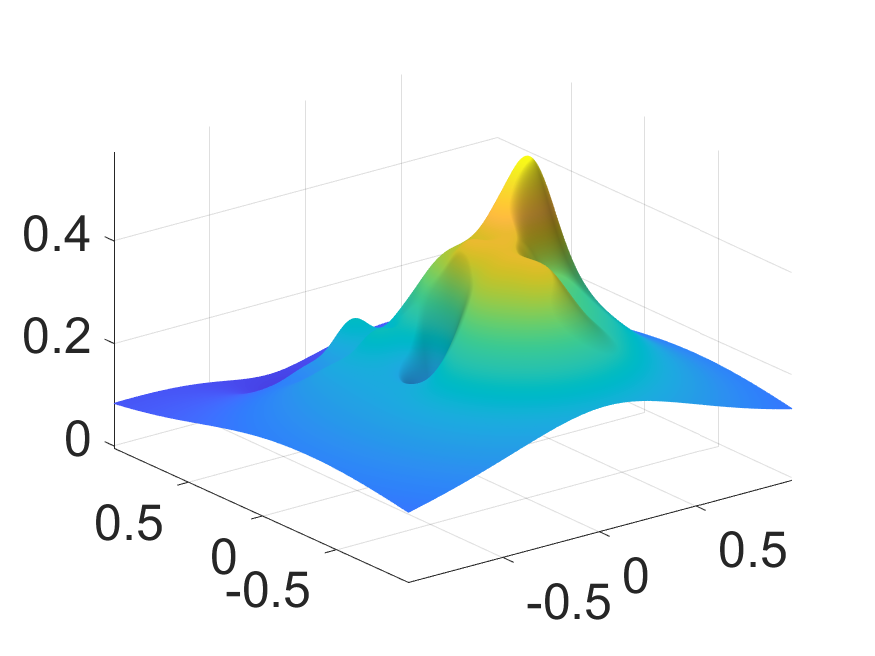}
  \includegraphics[width=0.32\textwidth]{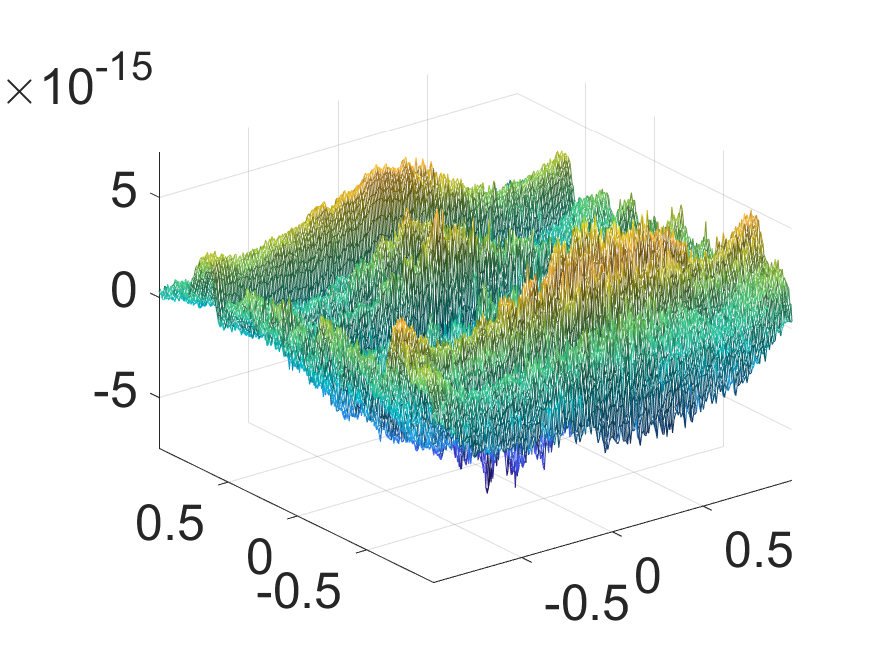}
  \includegraphics[width=0.32\textwidth]{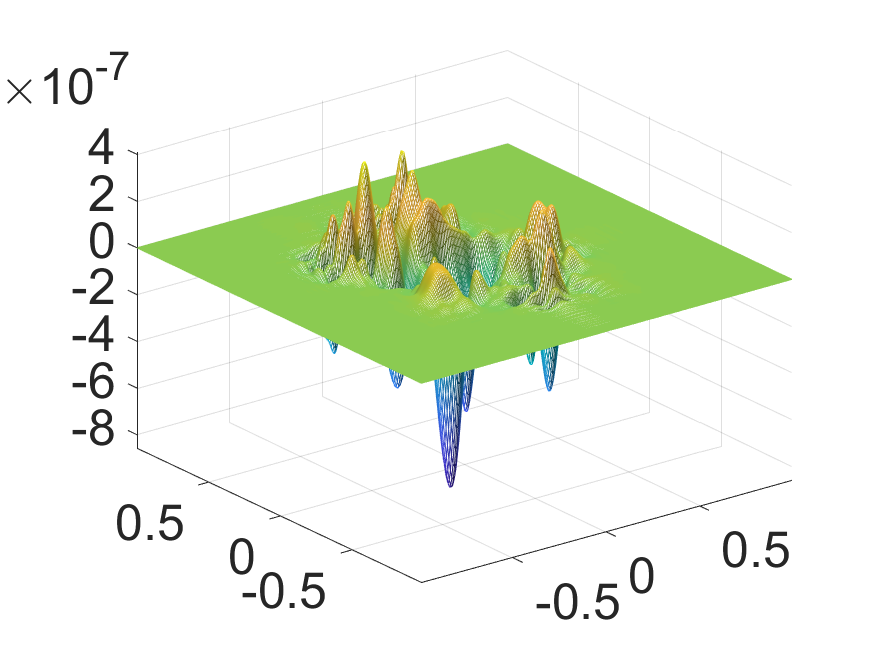}
  \caption{{\bf Left}: the middle slice $\hat{f}_{\bf m}(:,:,t_{n/2}) \approx {\bf A}_{R_l}(:,:,n/2)$ of the degree $m=129$ ChebTuck format of a truncated Fasciculin 1 fragment with $N=500$ and $n=256$. {\bf Middle}: the error during RHOSVD compression, i.e. $\hat{f}_{\bf m}(:,:,t_{n/2}) - \tilde{f}_{\bf m}(:,:,t_{n/2})$. {\bf Right}: the total error $\hat{f}_{\bf m}(:,:,t_{n/2}) - {\bf A}_{R_l}(:,:,n/2)$.
  }
  \label{fig:chebfun3h_error}
\end{figure}
\begin{figure}
  \centering
  \includegraphics[width=0.32\textwidth]{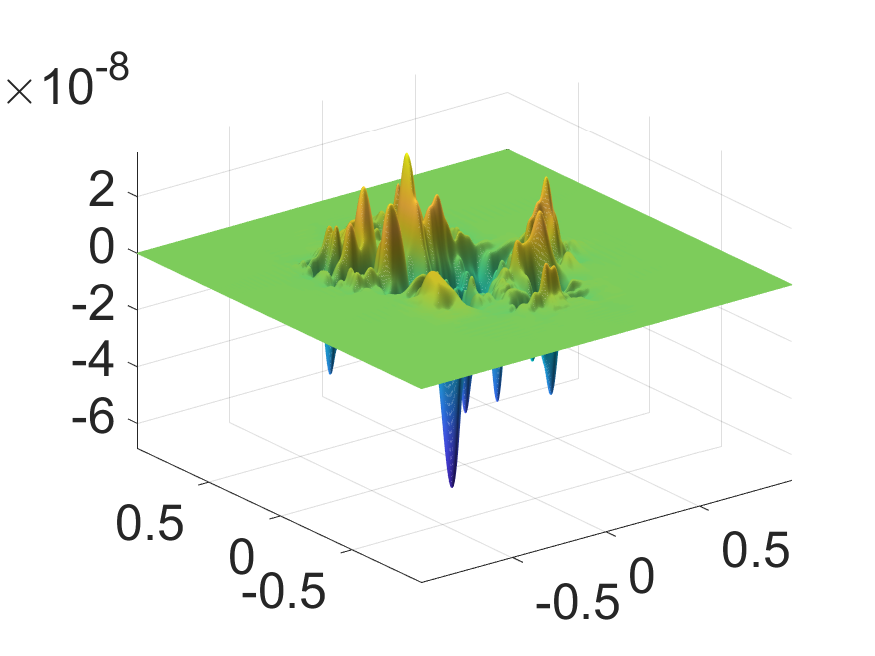}
  \includegraphics[width=0.32\textwidth]{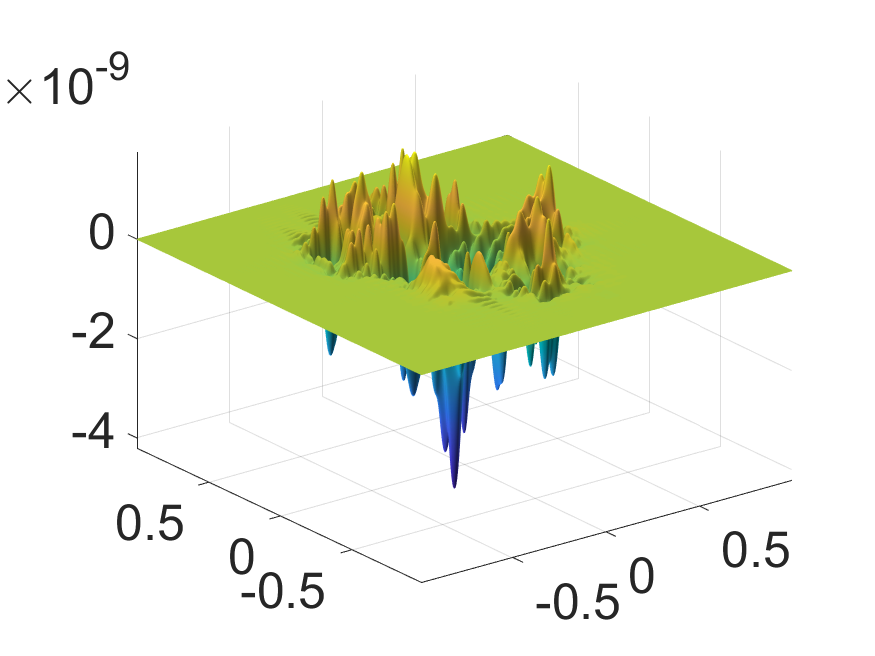}
  \includegraphics[width=0.32\textwidth]{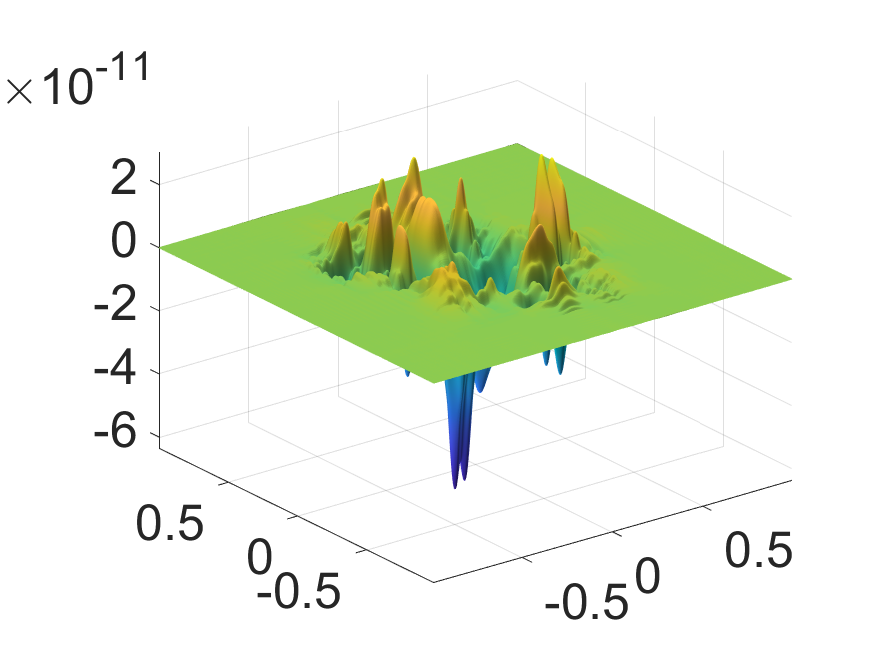}
  \caption{The degree $m=129$ ChebTuck approximation total error $\hat{f}_{\bf m}(:,:,t_{n/2}) - {\bf A}_{R_l}(:,:,n/2)$ of a truncated Fasciculin 1 fragment for grid sizes $n = 512, 1024, 2048$ from {\bf left} to {\bf right}.}
  \label{fig:chebfun3h_error_n}
\end{figure}
\begin{figure}[h]
  \centering
  \includegraphics[width=0.32\textwidth]{figures/Pn_total_err_n2048.png}
  \includegraphics[width=0.32\textwidth]{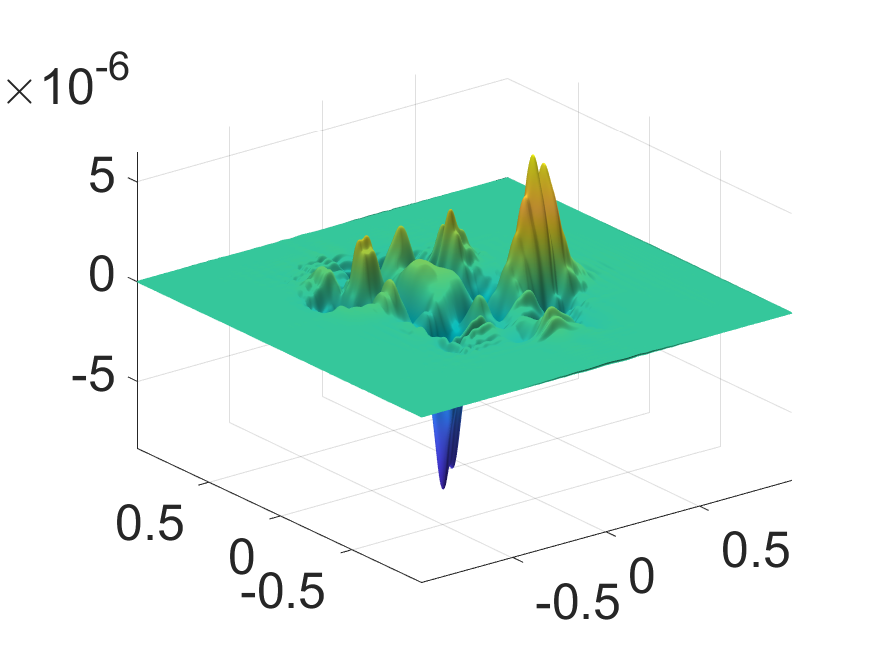}
  \includegraphics[width=0.32\textwidth]{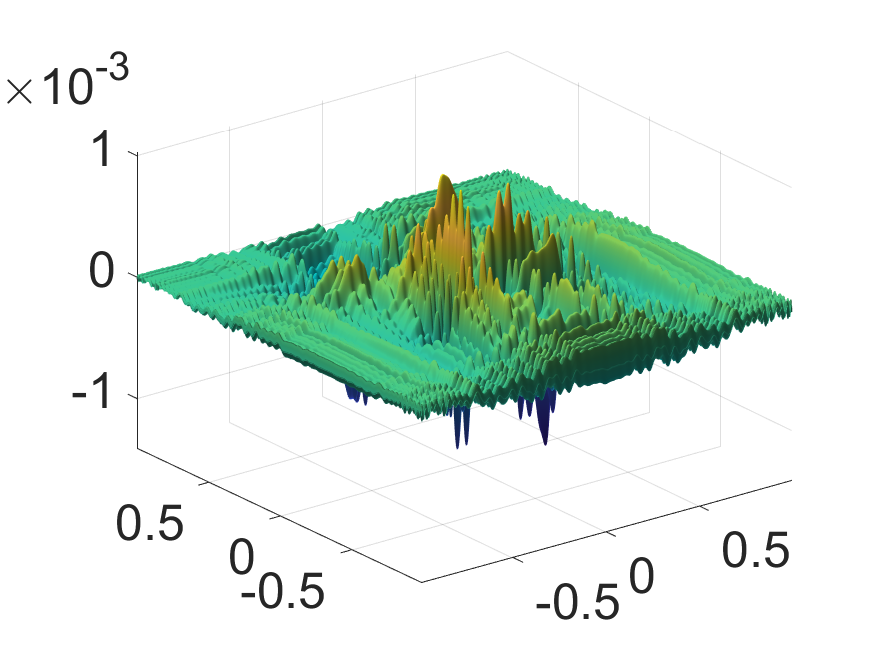}
  \caption{The total error $\hat{f}_{\bf m}(:,:,t_{n/2}) - {\bf A}_{R_l}(:,:,n/2)$ of the degree $m=129$ ChebTuck approximation of a truncated Fasciculin 1 fragment for grid size $n=2048$ with different univariate interpolations in \cref{eq:cp_2_spline}. {\bf Left}: cubic spline, {\bf middle}: piecewise linear, {\bf right}: nearest neighbor.}
  \label{fig:chebfun3h_error_nn}
\end{figure}

\subsubsection{\texorpdfstring{$O(\log N)$}{logN} scaling of the Tucker rank of the ChebTuck format}
\label{sec:logscaling}

In this section, we demonstrate numerically that the Tucker rank of our mesh-free ChebTuck format of the long-range part of the total potential for truncated Fasciculin 1 fragments scales logarithmically with the number of particles $N$, consistent with discussions in \cref{sec:cheb_tuck_error_cp_long_range} as well as the result for grid-based range-separated format as in \cite[Theorem 3.1]{benner2018range}. We consider fragment systems in $d=3$ dimensions with $N$ atoms discretized on a Cartesian grid with size $n = 2048$. For varying fragment size $N = 100, \ldots, 600$ and Chebyshev degree $m = 9, \ldots, 1025$, we show the Tucker rank of the ChebTuck format with $\varepsilon = 10^{-7}$ RHOSVD truncation tolerance in \cref{tab:tucker_rank_total}. We can see that the Tucker rank grows only mildly with respect to the number of atoms $N$. To have a closer look at the actual behavior, we also plot the singular values of the mode-$2$ side matrices of the Chebyshev coefficient tensor ${\bf C}$ (which is in CP format by construction) for $m = 129$ in \cref{fig:svals_CRS_n_256_cheb_deg_129}. We can see that the singular values drop to the level of machine precision after the first few significant ones, which also explains the behavior of the middle plot in \cref{fig:chebfun3h_error}.

\begin{table}[h]
  \centering
  \begin{tabular}{ccccccccc}
    \toprule
    \diagbox{$N$}{$m$} & 9 & 17 & 33 & 65 & 129 & 257 & 513 & 1025 \\
    \midrule
    100                & 9 & 17 & 26 & 26 & 26  & 26  & 26  & 26   \\
    200                & 9 & 17 & 30 & 30 & 30  & 30  & 30  & 30   \\
    300                & 9 & 17 & 31 & 32 & 32  & 32  & 32  & 32   \\
    400                & 9 & 17 & 31 & 32 & 32  & 32  & 32  & 32   \\
    500                & 9 & 17 & 31 & 32 & 32  & 32  & 32  & 32   \\
    600                & 9 & 17 & 31 & 33 & 33  & 33  & 33  & 33   \\
    \bottomrule
  \end{tabular}
  \caption{Tucker rank $\max_{\ell=1,2,3}{r_{\ell}}$ of the ChebTuck format with RHOSVD truncation tolerance $10^{-7}$ for the long-range part of the total potential in truncated Fasciculin 1 fragments with $n=2048$ for different particle numbers $N$ and Chebyshev degrees $m$.}
  \label{tab:tucker_rank_total}
\end{table}

This logarithmic-in-$N$ rank behavior is the key reason why the long-range stage can be represented by a moderate Chebyshev degree and a compact Tucker core, instead of carrying large full-grid objects into subsequent multilinear operations.
The dramatic compression benefit is illustrated in \cref{tab:compression_ratio}, which reports explicit parameter counts at the practical working degree $m=129$. In all tested particle ranges, ChebTuck uses only about $1.0\%$--$1.4\%$ of the CP parameters, i.e., about $98.6\%$--$99.0\%$ savings, while preserving high accuracy (error at the level of $10^{-7}$ as shown throughout these experiments).
From the same experiments, even for the much larger degree $m=1025$, the parameter savings remain above $90\%$, but $m=129$ is already sufficient for the target chemical-accuracy-level computations considered here.

\begin{table}[h]
  \centering
  \begin{tabular}{ccccc}
    \toprule
    $N$ & CP params & ChebTuck params & Compression & Savings \\
    \midrule
    100 & 2,064,384 & 20,238 & 1.0\% & 99.0\% \\
    200 & 2,506,752 & 29,040 & 1.2\% & 98.8\% \\
    300 & 2,586,624 & 35,031 & 1.4\% & 98.6\% \\
    400 & 2,875,392 & 38,234 & 1.3\% & 98.7\% \\
    500 & 3,121,152 & 40,540 & 1.3\% & 98.7\% \\
    600 & 3,348,480 & 45,054 & 1.3\% & 98.7\% \\
    \bottomrule
  \end{tabular}
  \caption{Parameter-count comparison between the CP input format ($n=2048$) and the ChebTuck representation ($m=129$) for truncated Fasciculin 1 long-range potentials. Compression is reported as $(\text{ChebTuck parameters})/(\text{CP parameters})$.}
  \label{tab:compression_ratio}
\end{table}

\begin{figure}[h]
  \centering
  \includegraphics[height=0.24\textwidth]{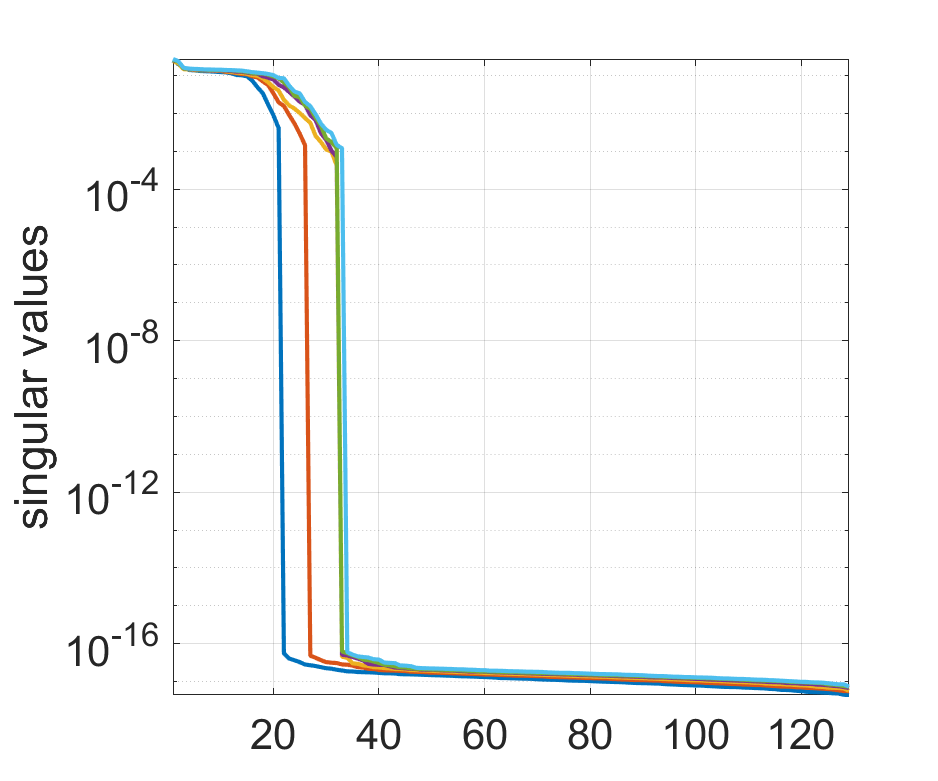}
  \includegraphics[height=0.24\textwidth]{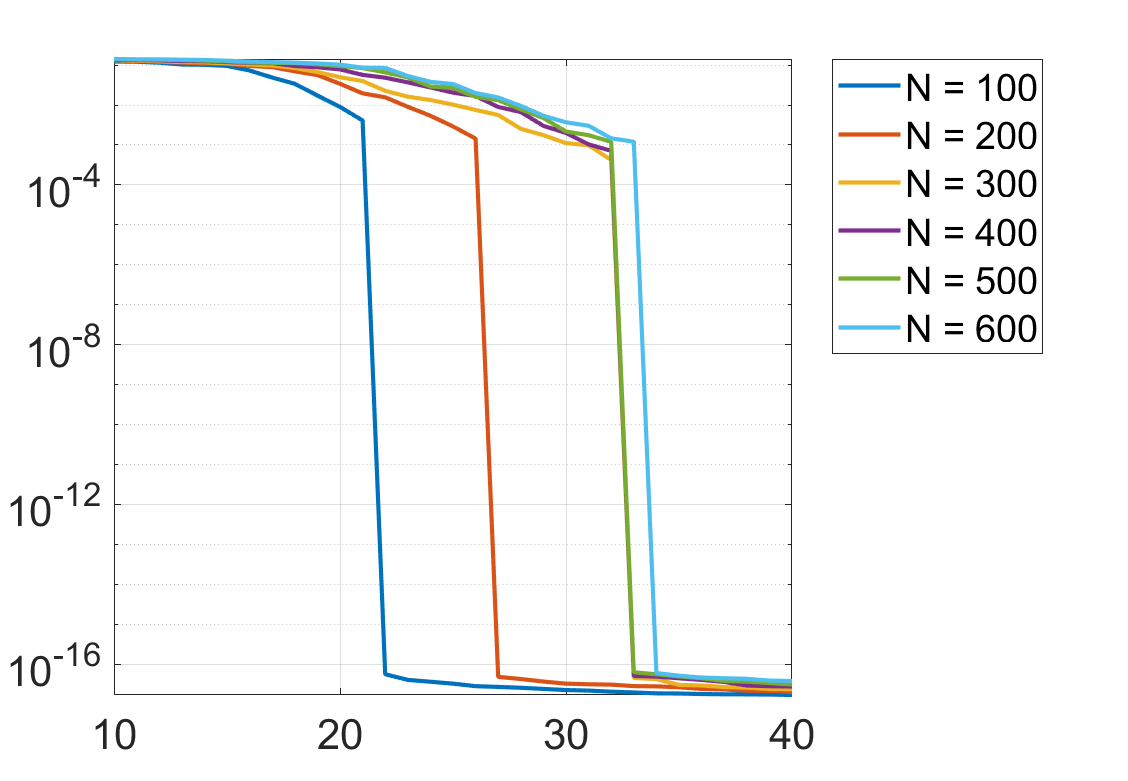}
  \caption{{\bf Left}: Singular values of the mode-2 side matrices of the Chebyshev coefficient tensor ${\bf C} \in \R^{m \times m \times m}$ with $m = 129$ for the long-range part of the total potential in truncated Fasciculin 1 fragments with $n=2048$. {\bf Right}: the zoom-in from the 10th to the 40th singular values of the same plot.}
  \label{fig:svals_CRS_n_256_cheb_deg_129}
\end{figure}

\begin{remark}\label{rem:long_range-ChebTuck_vs_FMM_MultPart}
Taking into account \cref{rem:long_range-ChebTuck_vs_FMM_Newton}, we notice that the numerical evidence in \cref{sec:cheb_tuck_fail,sec:rs_single_newton,sec:RS_tensor} indicates the efficiency of ChebTuck
approximation applied to the long-range component in the singular Newton potential, as well as in the multi-particle potential, both defined globally in the whole computational domain, while it fails if applied to the
complete singular potential, as expected. Here we realize that opposite to our approach,
in FMM-type methods the so-called ``long-range'' part of the potential is defined only on small geometric clusters located on some distance from the singularity (i.e., particle center).
This leads to, e.g. in BEM (boundary element method) and in the H-matrix techniques, the complicated patch-wise averaging of point charges based on rather cumbersome hierarchical clustering procedures of the computational domain, to be
implemented for each particle of the system, providing the approximation of total potential only at the centers of particles.
Our global ChebTuck approximation on the fine spatial grid can be further used in numerical PDE applications, for approximation of scattering data, as well as for efficient calculation of various physical parameters of multi-particle system.
\end{remark}

\subsubsection{A lattice-type compound with vacancies}\label{sec:lattice}

 Finally, we apply our hybrid tensor format ChebTuck to the problem
of low-parametric approximation of electrostatic potentials of large lattice-type structures, see
\cite{khoromskaia2014grid,Khor_bookQC_2018} for tensor-based approach to summation of electrostatic potentials on
orthogonal lattices.
Fast computation of lattice potentials was since longer an important topic of numerical analysis because of various applications
in crystal modeling, in quantum chemistry computations, in
polymer simulations and many other fields.
Traditional numerical methods for lattice summation are based on
classical fast Ewald summation method via transition of the problem
to the Fourier space \cite{Ewald:27,CanLe:13,HuMcCam:1999},
or, alternatively, by using the fast multipole method
\cite{greengard1987fast}.

Here we demonstrate our new approach in application to the long-range part of the many-particles lattice-structured electrostatic potential.
In this case, the charge centers ${\bf x}_v$ in \cref{eq:total_potential} are located on a regular grid.
For a lattice-type structure of size $24 \times 24 \times 4$ with $N = 2304$ particles, the canonical representation of its potential is given by a CP tensor of size $n\times n \times n$ with $n = 960$ and rank $R = 50$.
We apply \cref{alg:cheb_tuck_alg_cp} to approximate the long-range part of the potential with Chebyshev degrees $m_1 = m_2 = m_3 = m = 129$ and Tucker-ALS truncation tolerance $10^{-7}$. The results are shown in \cref{fig:lattice_error_long_range}. We observe that the error is of magnitude $10^{-6}$, which is consistent with the prescribed tolerance and the previous results on the biomolecular fragment examples.
The numerical results confirm the efficiency of the new mesh-less hybrid
tensor decomposition on the example of complicated multi-particle potential
arising in quantum chemistry, biomolecular modeling, etc.
Indeed, it demonstrates the same approximation error as for the
established grid-based tensor approximation method applied on relatively
large grids, required for the good resolution of complicated multi-particle potentials.

\begin{figure}[h]
  \centering
  \includegraphics[width=0.32\textwidth]{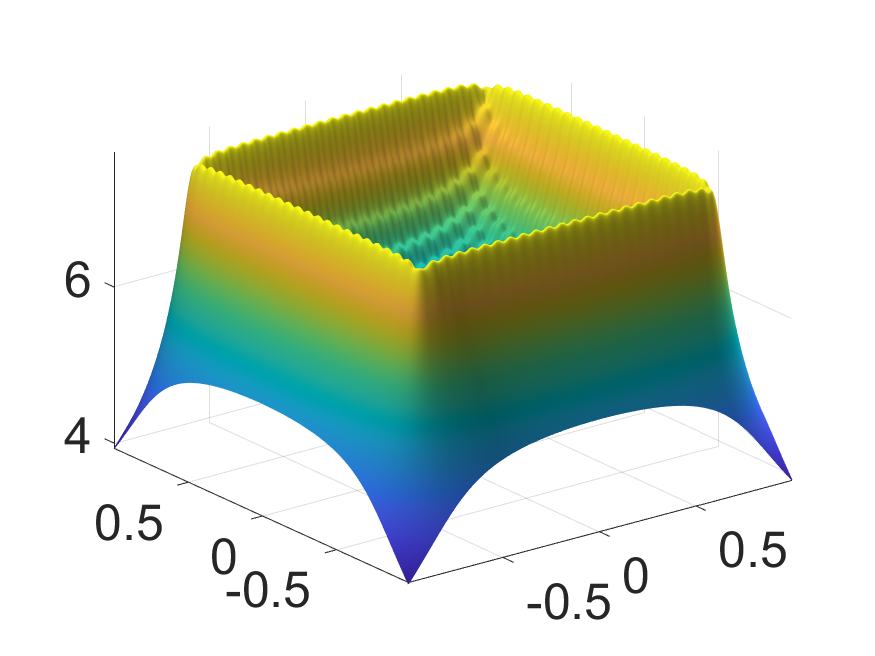}
  \includegraphics[width=0.32\textwidth]{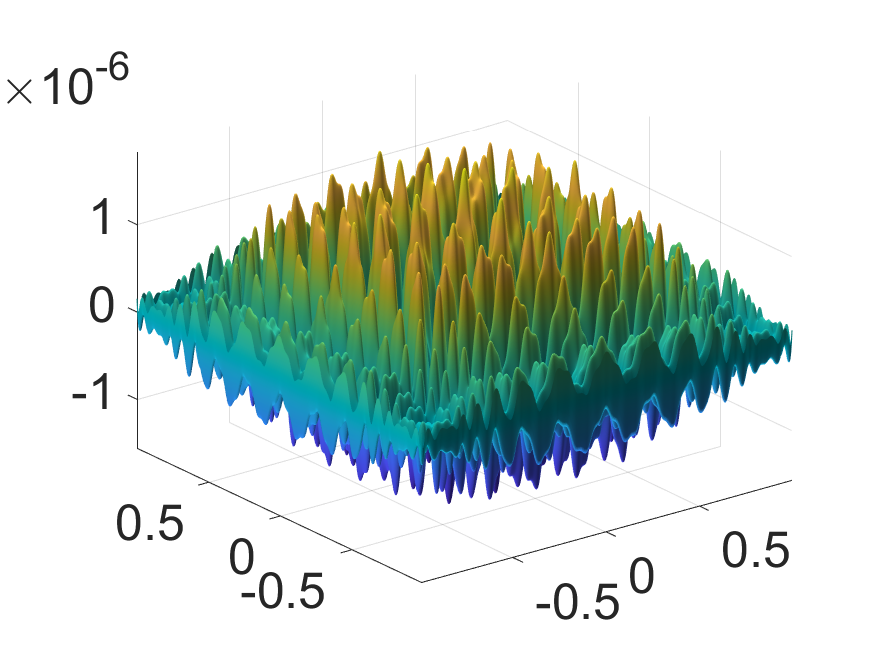}
  \includegraphics[width=0.32\textwidth]{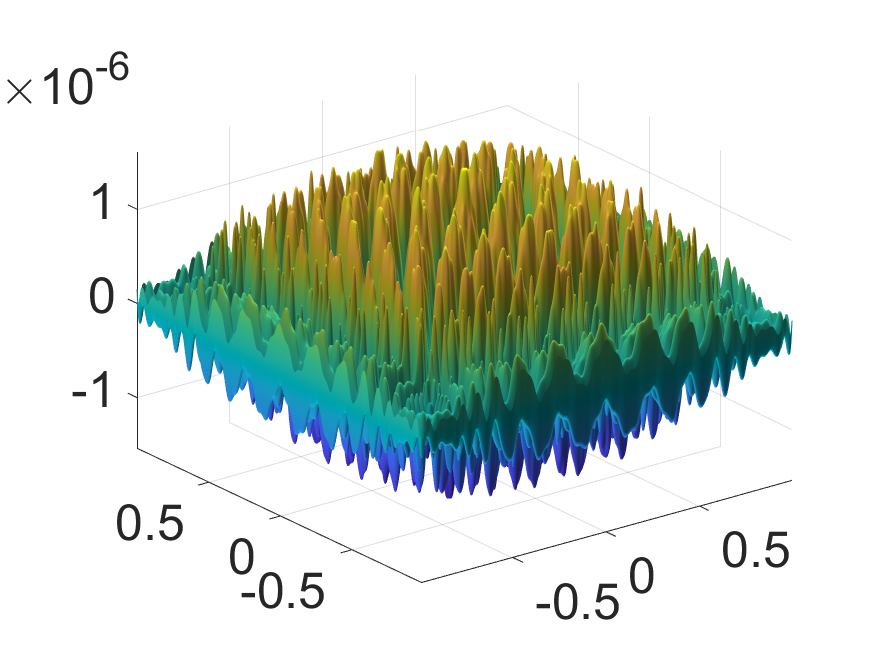}
  \caption{{\bf Left}: the middle slice $\hat{f}_{\bf m}(:,:,t_{n/2}) \approx {\bf A}_{R_l}(:,:,n/2)$ of the degree $m=129$ ChebTuck format for the long-range part of a lattice structure with $N=2304$ and $n=960$. {\bf Middle}: the error during RHOSVD compression, i.e. $\hat{f}_{\bf m}(:,:,t_{n/2}) - \tilde{f}_{\bf m}(:,:,t_{n/2})$. {\bf Right}: the total error $\hat{f}_{\bf m}(:,:,t_{n/2}) - {\bf A}_{R_l}(:,:,n/2)$.}
  \label{fig:lattice_error_long_range}
\end{figure}
Along the line of
\cref{rem:long_range-ChebTuck_vs_FMM_MultPart}
and \cref{rem:long_range-ChebTuck_vs_FMM_Newton}, here we observe that numerical illustrations in this section demonstrate that the proposed new technique achieves the same precision for calculated long-range lattice potential as the well established grid based tensor methods, but with lower computational and storage costs.
\section{Conclusions}\label{sec:conclude}

We introduce and analyze a mesh-free two-level hybrid Tucker tensor format
for approximation of trivariate functions on a hypercube with both high and reduced regularity, which combines the tensor product Chebyshev polynomial interpolation in the global computational domain
with the ALS-based Tucker decomposition of the Chebyshev coefficient tensor.
For functions with multiple cusps (or even stronger singularities), our method applies to the long-range component of the target function represented in low-rank RS tensor format.
The Tucker rank optimization via ALS-type rank-structured decomposition of the coefficient tensor leads to nearly optimal $\varepsilon$-rank Tucker decomposition of the initial function with merely the same rank parameter as for the alternative grid-based methods.
We present numerous numerical tests demonstrating the efficiency of the proposed techniques on the examples of many-particle electrostatic potentials in  bio-molecular systems and in lattice-structured compounds. In case of input functions presented in the canonical (CP) tensor format, our method needs only 1D-Chebyshev interpolations of the 1D-canonical vectors.

Finally, we notice that the presented techniques can be applied to the Green functions for various elliptic operators and can be extended to higher dimensions. In the latter case, the combination with cross-based tensor decompositions in TT, QTT or H-Tucker formats is necessary to avoid the curse of dimensionality.

\section*{Reproducibility statement}
The full source code and data to reproduce the numerical experiments is available at \url{https://github.com/bonans/ChebTuck}.
\bibliographystyle{siamplain}
\bibliography{references}

\begin{thebibliography}{10}

\bibitem{Bach:23}
{\sc M.~Bachmayr}, {\em Low-rank tensor methods for partial differential
  equations}, Acta Numer., 32 (2023), pp.~1--121.

\bibitem{Usch:16}
{\sc M.~Bachmayr, R.~Schneider, and A.~Uschmajew}, {\em Tensor networks and
  hierarchical tensors for the solution of high-dimensional partial
  differential equations}, Found. Comput. Math., 16 (2016), pp.~1423--1472.

\bibitem{Beben:11}
{\sc M.~Bebendorf}, {\em Adaptive cross approximation of multivariate
  functions}, Constr. Approx., 34 (2011), pp.~149--179.

\bibitem{BenFs:2024}
{\sc P.~Benner and H.~Fa\ss{}bender}, {\em Modellreduktion: Eine
  systemtheoretisch orientierte Einf\"{u}hrung}, Springer Verlag, Berlin, 2024.

\bibitem{PBE_RS:21}
{\sc P.~Benner, V.~Khoromskaia, B.~Khoromski, C.~Kweyu, and S.~M.}, {\em
  Regularization of {P}oisson--{B}oltzmann type equations with singular source
  terms using the range-separated tensor format}, SIAM J. Sci. Comput., 43
  (2021), pp.~A415--A445.

\bibitem{benner2018range}
{\sc P.~Benner, V.~Khoromskaia, and B.~N. Khoromskij}, {\em Range-separated
  tensor format for many-particle modeling}, SIAM J. Sci. Comput., 40 (2018),
  pp.~A1034--A1062.

\bibitem{RSDock:23}
{\sc P.~Benner, B.~N. Khoromskij, V.~Khoromskaia, and M.~Stein}, {\em Fast
  tensor-based electrostatic energy calculations in the perspective of
  protein-ligand docking problem}, arXiv preprint arXiv:2510.26611,  (2025).

\bibitem{benner2017model}
{\sc P.~Benner, M.~Ohlberger, A.~Cohen, and K.~Willcox}, {\em Model {R}eduction
  and {A}pproximation: {T}heory and {A}lgorithms}, SIAM, 2017.

\bibitem{Ber:88}
{\sc J.-P. Berrut}, {\em Rational functions for guaranteed and experimentally
  well-conditioned global interpolation}, Comput. Math. Appl., 15 (1988),
  pp.~1--16.

\bibitem{BerTre}
{\sc J.-P. Berrut and L.~N. Trefethen}, {\em Barycentric {L}agrange
  interpolation}, SIAM Rev., 46 (2004), pp.~501--517.

\bibitem{bertoglio2012low}
{\sc C.~Bertoglio and B.~N. Khoromskij}, {\em Low-rank quadrature-based tensor
  approximation of the {G}alerkin projected {N}ewton/{Y}ukawa kernels}, Comput.
  Phys. Commun., 183 (2012), pp.~904--912.

\bibitem{townsend2020ballfun}
{\sc N.~Boull\'e and A.~Townsend}, {\em Computing with functions in the ball},
  SIAM J. Sci. Comput., 42 (2020), pp.~C169--C191.

\bibitem{Braess:86}
{\sc D.~Braess}, {\em Nonlinear {A}pproximation {T}heory}, vol.~7 of Springer
  Series in Computational Mathematics, Springer-Verlag, Berlin, 1986.

\bibitem{CanLe:13}
{\sc E.~Canc\`es and C.~Le~Bris}, {\em Mathematical modeling of point defects
  in materials science}, Math. Models Methods Appl. Sci., 23 (2013),
  pp.~1795--1859.

\bibitem{che2019randomized}
{\sc M.~Che and Y.~Wei}, {\em Randomized algorithms for the approximations of
  {T}ucker and the tensor train decompositions}, Adv. Comput. Math., 45 (2019),
  pp.~395--428.

\bibitem{Cheb:1859}
{\sc P.~L. Chebyshev}, {\em Sur les questions de minima qui se rattachent \'{a}
  la representation approximative des fonctions}, Mem. Acad. Sci. P\'{e}tersb.,
  7 (1859).

\bibitem{de1978practical}
{\sc C.~de~Boor}, {\em A {P}ractical {G}uide to {S}plines}, Springer-Verlag,
  New York-Berlin, 1978.

\bibitem{de2000multilinear}
{\sc L.~De~Lathauwer, B.~De~Moor, and J.~Vandewalle}, {\em A multilinear
  singular value decomposition}, SIAM J. Matrix Anal. Appl., 21 (2000),
  pp.~1253--1278.

\bibitem{de2000best}
{\sc L.~De~Lathauwer, B.~De~Moor, and J.~Vandewalle}, {\em On the best rank-1
  and rank-{$(R_1,R_2,\cdots,R_N)$} approximation of higher-order tensors},
  SIAM J. Matrix Anal. Appl., 21 (2000), pp.~1324--1342.

\bibitem{Kress:21}
{\sc S.~Dolgov, D.~Kressner, and C.~Str\"ossner}, {\em Functional {T}ucker
  approximation using {C}hebyshev interpolation}, SIAM J. Sci. Comput., 43
  (2021), pp.~A2190--A2210.

\bibitem{DS:2014}
{\sc S.~V. Dolgov and D.~V. Savostyanov}, {\em Alternating minimal energy
  methods for linear systems in higher dimensions}, SIAM J. Sci. Comput., 36
  (2014), pp.~A2248--A2271.

\bibitem{Chebfun:2014}
{\sc T.~A. Driscoll, N.~Hale, and L.~N. Trefethen}, {\em Chebfun {G}uide},
  Pafnuty Publications, Oxford, 2014.

\bibitem{Ewald:27}
{\sc P.~Ewald}, {\em Die berechnung optische und elektrostatischer
  gitterpotentiale}, Annalen der Physik, 369 (1921), pp.~253--287.

\bibitem{Gao:22}
{\sc Z.~Gao, J.~Liang, and Z.~Xu}, {\em A kernel-independent
  sum-of-exponentials method}, J. Sci. Comput., 93, 40 (2022).

\bibitem{Green:18}
{\sc L.~Greengard, S.~Jiang, and Y.~Zhang}, {\em The anisotropic truncated
  kernel method for convolution with free-space {G}reen's functions}, SIAM J.
  Sci. Comput., 40 (2018), pp.~A3733--A3754.

\bibitem{greengard1987fast}
{\sc L.~Greengard and V.~Rokhlin}, {\em A fast algorithm for particle
  simulations}, J. Comput. Phys., 73 (1987), pp.~325--348.

\bibitem{HaKh:Lin}
{\sc W.~Hackbusch and B.~N. Khoromskij}, {\em A sparse {$\mathscr{H}$}-matrix
  arithmetic. {II}. {A}pplication to multi-dimensional problems}, Computing, 64
  (2000), pp.~21--47.

\bibitem{hackbusch2006low}
{\sc W.~Hackbusch and B.~N. Khoromskij}, {\em Low-rank {K}ronecker-product
  approximation to multi-dimensional nonlocal operators. {P}art {I}.
  {S}eparable approximation of multi-variate functions}, Computing, 76 (2006),
  pp.~177--202.

\bibitem{Trefeth:2017}
{\sc B.~Hashemi and L.~N. Trefethen}, {\em Chebfun in three dimensions}, SIAM
  J. Sci. Comput., 39 (2017), pp.~C341--C363.

\bibitem{Hitch:27}
{\sc F.~L. Hitchcock}, {\em The expression of a tensor or a polyadic as a sum
  of products}, J. Math. and Phys., 6 (1927), pp.~164--189.

\bibitem{HuMcCam:1999}
{\sc P.~H. H{\"u}nenberger and J.~A. McCammon}, {\em Effect of artificial
  periodicity in simulations of biomolecules under {E}wald boundary conditions:
  a continuum electrostatics study}, Biophys. Chemistry, 78 (1999), pp.~69--88.

\bibitem{kazeev2013low}
{\sc V.~Kazeev, O.~Reichmann, and C.~Schwab}, {\em Low-rank tensor structure of
  linear diffusion operators in the {TT} and {QTT} formats}, Linear Algebra
  Appl., 438 (2013), pp.~4204--4221.

\bibitem{khan2025parametricHmatrix}
{\sc A.~Khan, C.~Chen, V.~Rao, and A.~K. Saibaba}, {\em Parametric hierarchical
  matrix approximations to kernel matrices}, arXiv preprint arXiv:2511.03109,
  (2025).

\bibitem{khan2025parametricTT}
{\sc A.~Khan and A.~K. Saibaba}, {\em Parametric kernel low-rank approximations
  using tensor train decomposition}, SIAM J. Matrix Anal. Appl., 46 (2025),
  pp.~1006--1036.

\bibitem{khoromskaia2014grid}
{\sc V.~Khoromskaia and B.~N. Khoromskij}, {\em Grid-based lattice summation of
  electrostatic potentials by assembled rank-structured tensor approximation},
  Comput. Phys. Commun., 185 (2014), pp.~3162--3174.

\bibitem{Khor_bookQC_2018}
{\sc V.~Khoromskaia and B.~N. Khoromskij}, {\em Tensor {N}umerical {M}ethods in
  {Q}uantum {C}hemistry}, De Gruyter, 2018.

\bibitem{khoromskaia2022ubiquitous}
{\sc V.~Khoromskaia and B.~N. Khoromskij}, {\em Ubiquitous nature of the
  reduced higher order {S}{V}{D} in tensor-based scientific computing}, Front.
  Appl. Math. Stat., 8:826988 (2022).

\bibitem{KhQuant:09}
{\sc B.~N. Khoromskij}, {\em {$O(d\log N)$}-quantics approximation of
  {$N$}-{$d$} tensors in high-dimensional numerical modeling}, Constr. Approx.,
  34 (2011), pp.~257--280.

\bibitem{Khor-book_2018}
{\sc B.~N. Khoromskij}, {\em Tensor {N}umerical {M}ethods in {S}cientific
  {C}omputing}, De Gruyter, 2018.

\bibitem{khoromskij2007low}
{\sc B.~N. Khoromskij and V.~Khoromskaia}, {\em Low rank {T}ucker-type tensor
  approximation to classical potentials}, Cent. Eur. J. Math., 5 (2007),
  pp.~523--550.

\bibitem{khoromskij2009multigrid}
{\sc B.~N. Khoromskij and V.~Khoromskaia}, {\em Multigrid accelerated tensor
  approximation of function related multidimensional arrays}, SIAM J. Sci.
  Comput., 31 (2009), pp.~3002--3026.

\bibitem{kressner2023streaming}
{\sc D.~Kressner, B.~Vandereycken, and R.~Voorhaar}, {\em Streaming tensor
  train approximation}, SIAM J. Sci. Comput., 45 (2023), pp.~A2610--A2631.

\bibitem{DuMaBoFo:92}
{\sc M.~Le~Du, P.~Marchot, P.~Bougis, and J.~Fontecilla-Camps}, {\em 1.9
  {A}ngstrom resolution structure of fasciculine 1, an
  anti-acetylcholinesterase toxin from green mamba snake venom}, J. Biol.
  Chem., 267 (1992), pp.~22122--30.

\bibitem{li2025hierarchical}
{\sc Y.~Li and J.~Liu}, {\em Hierarchical tucker low-rank matrices:
  Construction and matrix-vector multiplication}, arXiv preprint
  arXiv:2508.05958,  (2025).

\bibitem{Gavini:21}
{\sc C.-C. Lin, P.~Motamarri, and V.~Gavini}, {\em Tensor-structured algorithm
  for reduced-order scaling large-scale {K}ohn-{S}ham density functional theory
  calculations}, npj Comput. Mater., 7, 50 (2021).

\bibitem{Solo:21}
{\sc L.~Ma and E.~Solomonik}, {\em Fast and accurate randomized algorithms for
  low-rank tensor decompositions}, Adv. Neural Inf. Process., 34 (2021),
  pp.~24299--24312.

\bibitem{Schwab:22}
{\sc C.~Marcati, M.~Rakhuba, and C.~Schwab}, {\em Tensor rank bounds for point
  singularities in {$\Bbb{R}^3$}}, Adv. Comput. Math., 48, 18 (2022).

\bibitem{mason1980near}
{\sc J.~C. Mason}, {\em Near-best multivariate approximation by {F}ourier
  series, {C}hebyshev series and {C}hebyshev interpolation}, J. Approx. Theory,
  28 (1980), pp.~349--358.

\bibitem{mason1982minimal}
{\sc J.~C. Mason}, {\em Minimal projections and near-best approximations by
  multivariate polynomial expansion and interpolation}, in Multivariate
  approximation theory, {II}, Birkh\"auser, Basel, 1982, pp.~241--254.

\bibitem{mason2002chebyshev}
{\sc J.~C. Mason and D.~C. Handscomb}, {\em Chebyshev {P}olynomials}, Chapman
  \& Hall/CRC, New York, 2002.

\bibitem{Osel_TT:11}
{\sc I.~V. Oseledets}, {\em Tensor-train decomposition}, SIAM J. Sci. Comput.,
  33 (2011), pp.~2295--2317.

\bibitem{oseledets2008tucker}
{\sc I.~V. Oseledets, D.~V. Savostianov, and E.~E. Tyrtyshnikov}, {\em Tucker
  dimensionality reduction of three-dimensional arrays in linear time}, SIAM J.
  Matrix Anal. Appl., 30 (2008), pp.~939--956.

\bibitem{rakhuba2015fast}
{\sc M.~V. Rakhuba and I.~V. Oseledets}, {\em Fast multidimensional convolution
  in low-rank tensor formats via cross approximation}, SIAM J. Sci. Comput., 37
  (2015), pp.~A565--A582.

\bibitem{MSchneid:22}
{\sc L.~Risthaus and M.~Schneider}, {\em Solving phase-field models in the
  tensor train format to generate microstructures of bicontinuous composites},
  Appl. Numer. Math., 178 (2022), pp.~262--279.

\bibitem{Kilmer:22}
{\sc A.~K. Saibaba, R.~Minster, and M.~E. Kilmer}, {\em Efficient randomized
  tensor-based algorithms for function approximation and low-rank kernel
  interactions}, Adv. Comput. Math., 48, 66 (2022).

\bibitem{Stenger:93}
{\sc F.~Stenger}, {\em Numerical {M}ethods {B}ased on {S}inc and {A}nalytic
  {F}unctions}, Springer-Verlag, New York, 1993.

\bibitem{Kress:24}
{\sc C.~Str\"ossner, B.~Sun, and D.~Kressner}, {\em Approximation in the
  extended functional tensor train format}, Adv. Comput. Math., 50, 54 (2024).

\bibitem{Gav:24}
{\sc V.~Subramanian, S.~Das, and V.~Gavini}, {\em Tucker tensor approach for
  accelerating {F}ock {E}xchange computations in a real-space finite-element
  discretization of generalized {K}ohn-{S}ham density functional theory}, J.
  Chem. Theory Comput., 20 (2024), pp.~3566--3579.

\bibitem{townsend2014thesis}
{\sc A.~Townsend}, {\em Computing with functions in two dimensions}, PhD
  thesis, University of Oxford, 2014.

\bibitem{townsend2013chebfun2}
{\sc A.~Townsend and L.~N. Trefethen}, {\em An extension of {C}hebfun to two
  dimensions}, SIAM J. Sci. Comput., 35 (2013), pp.~C495--C518.

\bibitem{townsend2015chebfun2}
{\sc A.~Townsend and L.~N. Trefethen}, {\em Continuous analogues of matrix
  factorizations}, Proc. A., 471 (2015), pp.~20140585, 21.

\bibitem{trefethen2019}
{\sc L.~N. Trefethen}, {\em Approximation {T}heory and {A}pproximation
  {P}ractice, {E}xtended {E}dition}, SIAM, 2019.

\bibitem{Tuck:1966}
{\sc L.~R. Tucker}, {\em Some mathematical notes on three-mode factor
  analysis}, Psychometrika, 31 (1966), pp.~279--311.

\bibitem{Yser:24}
{\sc H.~Yserentant}, {\em An iterative method for the solution of
  {L}aplace-like equations in high and very high space dimensions}, Numer.
  Math., 156 (2024), pp.~777--811.

\end{thebibliography}
\end{document}